\documentclass[11pt,onecolumn]{article}

\usepackage[T1]{fontenc}
\usepackage[utf8]{inputenc}
\usepackage{lmodern}
\usepackage{microtype}
\usepackage[margin=1.8cm,top=2cm,bottom=2.5cm]{geometry}
\usepackage{amsmath,amssymb,amsthm}
\usepackage{mathtools}
\usepackage{graphicx}
\usepackage{caption}
\usepackage{subcaption}
\usepackage{booktabs}
\usepackage{array}
\usepackage[colorlinks=true,citecolor=blue,urlcolor=blue,linkcolor=blue]{hyperref}
\usepackage[numbers,sort&compress]{natbib}
\usepackage{setspace}
\usepackage{parskip}
\usepackage{titlesec}
\usepackage{abstract}
\usepackage{tcolorbox}
\usepackage{comment}

\usepackage[colorlinks=true,citecolor=blue,urlcolor=blue,linkcolor=blue]{hyperref}
\usepackage[numbers,sort&compress]{natbib}
\usepackage{enumerate}

\usepackage[dvipsnames]{xcolor}

\usepackage{tikz}
\usetikzlibrary{arrows.meta,shapes.geometric,positioning,calc}

\definecolor{panelgray}{RGB}{240,242,245}
\definecolor{panelblue}{RGB}{219,234,254}
\definecolor{paneteal}{RGB}{204,251,241}
\definecolor{panelamb}{RGB}{254,243,199}
\definecolor{panelred}{RGB}{254,226,226}
\definecolor{cnavy}{RGB}{27,46,75}
\definecolor{cblue}{RGB}{37,99,235}
\definecolor{cteal}{RGB}{13,148,136}
\definecolor{camb}{RGB}{217,119,6}
\definecolor{cred}{RGB}{220,38,38}
\definecolor{cgray}{RGB}{107,114,128}

\titleformat{\section}{\bfseries\normalsize}{\thesection.}{0.5em}{}
\titleformat{\subsection}{\bfseries\small}{\thesubsection.}{0.5em}{}
\titlespacing{\section}{0pt}{8pt}{4pt}
\titlespacing{\subsection}{0pt}{6pt}{3pt}

\newcommand{\dd}{\mathrm{d}}
\newcommand{\Rbb}{\mathbb{R}}
\newcommand{\calN}{\mathcal{N}}
\newcommand{\calE}{\mathcal{E}}
\newcommand{\calP}{\mathcal{P}}
\newcommand{\tr}{\operatorname{Tr}}
\newcommand{\mfeff}{(\mathrm{eff})}
\newcommand{\ptar}{p^{(\mathrm{tar})}}
\newcommand{\pin}{p^{(\mathrm{in})}}

\newcommand{\sstar}{u^{(*)}}
\newcommand{\MF}{\mathrm{MF}}
\newcommand{\IA}{\mathrm{IA}}
\newcommand{\Veff}{V^{(\mathrm{eff})}}

\newtheorem{theorem}{Theorem}[section]
\newtheorem{proposition}[theorem]{Proposition}

\newtheorem{corollary}[theorem]{Corollary}
\newtheorem{remark}{Remark}[section]

\title{%
  \textbf{Mean-Field Path-Integral Diffusion:\\
  From Samples to Interacting Agents}
}

\author{%
  Michael (Misha) Chertkov\\[2pt]
  \small Graduate Interdisciplinary Program in Applied Mathematics\\
  \small \& Department of Mathematics, University of Arizona, Tucson, AZ, USA\\[2pt]
  \small Correspondence: \href{mailto:chertkov@arizona.edu}{chertkov@arizona.edu}
}

\date{\today}

\begin{document}
\maketitle

\begin{abstract}
Independent sample generation is the prevailing paradigm in modern diffusion-based generative models of AI. We ask a different question: can samples \emph{coordinate} through shared population statistics to transport probability mass more efficiently? We introduce Mean-Field Path-Integral Diffusion (MF-PID), a framework in which samples are promoted to interacting agents whose drift depends self-consistently on the evolving population density. The coupling converts distribution matching into a McKean--Vlasov extension of the stochastic optimal transport problem, unifying generative modeling and multi-agent control under the same Hamilton--Jacobi--Bellman/Kolmogorov--Fokker--Planck duality. We identify two analytically tractable regimes: a Linear--Quadratic--Gaussian (LQG) benchmark in which the infinite-dimensional mean-field system reduces to a finite set of Riccati and linear ODEs, and a Gaussian-mixture regime governed by a piecewise-constant protocol that preserves closed-form solvability. For a quadratic interaction potential with schedule $\beta_t$ and zero base drift we prove that the self-consistent MF guidance is the \emph{exact} linear interpolant between initial and target global means --- a result that holds for arbitrary initial and target densities and any $\beta_t$. Applied to demand-response control of energy systems, where agents aggregated into an ensemble are energy consumers (e.g.\ thermal zones within a building), MF-PID achieves 19--24\% reductions in cumulative control energy over independent-agent baselines while matching the prescribed terminal distribution exactly, and reveals how coordination redistributes actuation effort across heterogeneous sub-populations. The energy saving is independent of the number of zones per building ($d=1$--$32$ tested), confirming that the linear guidance formula broadcasts a single $d$-vector with $\mathcal{O}(d)$ communication and grows mildly in compute (sub-cubically for $d\leq 32$, asymptotically $\mathcal{O}(d^3)$ for $d\gg 1$).
\end{abstract}

\begin{figure*}[t]
\centering
\begin{tikzpicture}[font=\small, x=1mm, y=1mm]

\def\TW{170}
\def\TH{62}

\def\PaL{0}    \def\PaR{35}
\def\PbL{45}   \def\PbR{80}
\def\PcL{90}   \def\PcR{125}
\def\PdL{135}  \def\PdR{170}

\pgfmathsetmacro{\PaMid}{(\PaL+\PaR)/2}    
\pgfmathsetmacro{\PbMid}{(\PbL+\PbR)/2}    
\pgfmathsetmacro{\PcMid}{(\PcL+\PcR)/2}    
\pgfmathsetmacro{\PdMid}{(\PdL+\PdR)/2}    

\pgfmathsetmacro{\ArA}{(\PaR+\PbL)/2}      
\pgfmathsetmacro{\ArB}{(\PbR+\PcL)/2}      
\pgfmathsetmacro{\ArC}{(\PcR+\PdL)/2}      

\fill[cnavy] (0,\TH) rectangle (\TW,\TH+7);
\node[text=white, font=\small\bfseries, anchor=center]
  at (\TW/2, \TH+3.5)
  {Mean-Field Path-Integral Diffusion (MF-PID) $=$ Sample/Agent-Coordinated Optimal Transport};

\fill[panelgray!60] (\PaL,0) rectangle (\PaR,\TH);
\draw[cgray, line width=0.7pt, rounded corners=2pt]
  (\PaL,0) rectangle (\PaR,\TH);

\node[cgray, font=\small\bfseries, anchor=north]
  at (\PaMid, \TH-0.5) {Independent Agents};

\node[cnavy, font=\scriptsize, anchor=center]
  at (\PaMid, \TH-8) {$dx_t=\text{\color{red} score }dt+dW_t$};

\begin{scope}[shift={(7,40)}]
  \fill[cgray!25]
    plot[domain=-5:5,samples=30]
      (\x,{6*exp(-0.5*(\x/2.0)^2)}) -- (5,0) -- (-5,0) -- cycle;
  \draw[cgray, line width=0.8pt]
    plot[domain=-5:5,samples=30](\x,{6*exp(-0.5*(\x/2.0)^2)});
  \node[cgray, font=\footnotesize, anchor=north] at (0,-1) {$p^{(\mathrm{in})}$};
\end{scope}

\foreach \yoff/\op in {0/0.80, 3/0.55, -3/0.55, 5.5/0.35, -5.5/0.35}{
  \draw[cgray, line width=0.5pt, opacity=\op]
    (12, 40+\yoff)
    .. controls (15, 40+\yoff+0.4) and (19, 40+\yoff-0.4) ..
    (23, 40+\yoff);
}

\begin{scope}[shift={(28,40)}]
  \fill[cgray!25]
    plot[domain=-4:4,samples=30]
      (\x,{5*exp(-0.5*((\x+1.5)/1.0)^2)+10*exp(-0.5*((\x-1.8)/0.7)^2)}
      ) -- (4,0) -- (-4,0) -- cycle;
  \draw[cgray, line width=0.8pt]
plot[domain=-4:4,samples=50]
  (\x,{5*exp(-0.5*((\x+1.5)/1.0)^2)+10*exp(-0.5*((\x-1.8)/0.7)^2)});
  \node[cgray, font=\footnotesize, anchor=north] at (0,-1) {$p^{(\mathrm{out})}$};
\end{scope}

\node[cred, font=\footnotesize, anchor=center]
  at (\PaMid, 31) {Path Integral Diffusion };

\node[cnavy, font=\scriptsize, anchor=center]
  at (\PaMid, 26) {score $=\nabla_x \log \psi_t(x)$};

\node[cnavy, font=\scriptsize, anchor=center]
  at (\PaMid, 21) {"Sch\"{o}dinger Bridge"  $+ V_t(x)$};

\node[cgray, font=\scriptsize, anchor=west] at (\PaL+1, 13)
  {$\calE^{\IA(0)}$};
\fill[cred!50, rounded corners=1pt](\PaL+11,9)rectangle(\PaR-2,13);
\node[cgray, font=\scriptsize, anchor=center]
  at (\PaMid, 4.5) {Scen.\,A: 31.3 \quad B: 17.2};

\node[cblue, font=\scriptsize\bfseries, anchor=south]
  at (\ArA, \TH/2+3) {MF};
\draw[cblue, ->, >=Stealth, line width=2.5pt]
  (\ArA-3.5, \TH/2) -- (\ArA+3.5, \TH/2);

\fill[panelblue!50] (\PbL,0) rectangle (\PbR,\TH);
\draw[cblue, line width=1pt, rounded corners=2pt]
  (\PbL,0) rectangle (\PbR,\TH);

\node[cblue, font=\small\bfseries, anchor=north]
  at (\PbMid, \TH-0.5) {Interacting Agents};

\node[cnavy, font=\scriptsize, anchor=center]
  at (\PbMid, \TH-6) {MF-PID = H-PID};

\node[cnavy, font=\scriptsize, anchor=center]
  at (\PbMid, \TH-10) {\color{red} with $\nu_t=\mathbb{E}[x_t]$};

\begin{scope}[shift={(51,40)}]
  \fill[cblue!20]
    plot[domain=-5:5,samples=30]
      (\x,{6*exp(-0.5*(\x/2.0)^2)}) -- (5,0) -- (-5,0) -- cycle;
  \draw[cblue!80, line width=0.8pt]
    plot[domain=-5:5,samples=30](\x,{6*exp(-0.5*(\x/2.0)^2)});
  \node[cblue, font=\footnotesize, anchor=north] at (0,-1) {$p^{(\mathrm{in})}$};
\end{scope}

\foreach \yoff/\op in {0/0.85, 2.5/0.6, -2.5/0.6, 5/0.4, -5/0.4}{
  \draw[cblue, line width=0.55pt, opacity=\op]
    (56, 40+\yoff)
    .. controls (58.5, 40+\yoff*0.55) and (60.5, 40.5) ..
    (63, 40)
    .. controls (65.5, 39.5) and (67.5, 40+\yoff*0.55) ..
    (70, 40+\yoff);
}

\filldraw[fill=white, draw=cblue, line width=1pt]
  (63, 40) ellipse (4 and 3.5);
\node[cblue, font=\footnotesize\itshape] at (63, 41.2) {$\rho_t$};
\node[cblue, font=\scriptsize] at (63, 38.5) {};

\begin{scope}[shift={(73,40)}]
  \fill[cblue!20]
    plot[domain=-4:4,samples=30]
      (\x,{5*exp(-0.5*((\x+1.5)/1.0)^2)+10*exp(-0.5*((\x-1.8)/0.7)^2)}
      ) -- (4,0) -- (-4,0) -- cycle;
  \draw[cblue, line width=0.8pt]
    plot[domain=-4:4,samples=50]
  (\x,{5*exp(-0.5*((\x+1.5)/1.0)^2)+10*exp(-0.5*((\x-1.8)/0.7)^2)});
  \node[cblue, font=\footnotesize, anchor=north] at (0,-1) {$p^{(\mathrm{out})}$};
\end{scope}

\node[cnavy, font=\scriptsize, anchor=center]
  at (\PbMid, 26) {$u^{(*)}_t = \nabla\!\log\psi_t(x,\rho_t)$};

  
\node[cnavy, font=\scriptsize, anchor=center]
  at (\PbMid, 16)
  {$V^{(\mathrm{eff})}_t(x)\!\!=\!\!\int\!V_t(x\!-\! y)\rho_t(y)\mathrm{d}y$};

\node[cblue, font=\scriptsize, anchor=west] at (\PbL+1, 11)
  {$\calE^{\MF}$};
\fill[cblue!70, rounded corners=1pt](\PbL+10,7)rectangle(\PbL+30,11);
\node[cteal, font=\scriptsize\bfseries, anchor=center]
  at (\PbMid, 3.5) {$\downarrow$\,11--23\% energy saved};

\node[cteal, font=\scriptsize\bfseries, anchor=south]
  at (\ArB, \TH/2+3) {Exact 
  };
\draw[cteal, ->, >=Stealth, line width=2.5pt]
  (\ArB-3.5, \TH/2) -- (\ArB+3.5, \TH/2);

\fill[paneteal!60](\PcL,0)rectangle(\PcR,\TH);
\draw[cteal, line width=0.9pt, rounded corners=2pt]
  (\PcL,0)rectangle(\PcR,\TH);

\fill[paneteal!50]
  (\PcL+1.5, \TH/2+2+5) rectangle (\PcR-1.5, \TH-9+5);
\draw[cteal, rounded corners=2pt, line width=0.7pt]
  (\PcL+1.5, \TH/2+2+5) rectangle (\PcR-1.5, \TH-9+5);

\node[cteal, font=\small\bfseries, anchor=north]
  at (\PcMid, \TH-6) {LQG};

\node[cnavy, font=\scriptsize, align=center, anchor=center]
  at (\PcMid, \TH/2+15)
  {Linear $f$, Quadratic $V$\\Gaussian $\ptar$\\ $\longrightarrow$ Riccati ODEs
  };

\fill[panelamb!50](\PcL+1.5, 3)rectangle(\PcR-1.5, \TH/2-2);
\draw[camb, rounded corners=2pt, line width=0.7pt]
  (\PcL+1.5, 3)rectangle(\PcR-1.5, \TH/2-2);
\node[camb, font=\small\bfseries, anchor=north]
  at (\PcMid, \TH/2-3) {Score is Explicit};
\node[cnavy, font=\scriptsize, anchor=north, align=center]
  at (\PcMid, \TH/2-10)
  {$p^{(init)}$ \& $p^{(tar)}$ are GMs\\
  + $\beta_t$ is PWC\\
  \\ $\Rightarrow$ $\psi_t$ is GM (in $x$)\\ \& interval-analytic in $t$};

\node[cred, font=\scriptsize\bfseries, anchor=south]
  at (\ArC, \TH/2+3) {Apply};
\draw[cred, ->, >=Stealth, line width=2.5pt]
  (\ArC-3.5, \TH/2) -- (\ArC+3.5, \TH/2);

\fill[panelred!60](\PdL,0)rectangle(\PdR,\TH);
\draw[cred, line width=1pt, rounded corners=2pt]
  (\PdL,0)rectangle(\PdR,\TH);

\node[cred, font=\small\bfseries, anchor=north]
  at (\PdMid, \TH-0.5) {Demand Response};

\node[cnavy, font=\footnotesize\bfseries, anchor=north]
  at (\PdMid, \TH-7) {TCL fleet};

\node[cnavy, font=\scriptsize\bfseries, anchor=west]
  at (\PdL+2, 49) {$\calE$ -- energy savings };

\node[cgray, font=\scriptsize, anchor=east] at (\PdL+8, 44.5) {A};
\fill[cred!40,  rounded corners=1pt](\PdL+8, 42)rectangle(\PdR-3, 46);
\fill[camb!60,  rounded corners=1pt](\PdL+8, 42)rectangle(\PdL+26, 46);
\fill[cblue!80, rounded corners=1pt](\PdL+8, 42)rectangle(\PdL+19, 46);

\node[cgray, font=\scriptsize, anchor=east] at (\PdL+8, 37) {B};
\fill[cred!40,  rounded corners=1pt](\PdL+8, 34.5)rectangle(\PdR-3, 38.5);
\fill[camb!60,  rounded corners=1pt](\PdL+8, 34.5)rectangle(\PdL+23, 38.5);
\fill[cblue!80, rounded corners=1pt](\PdL+8, 34.5)rectangle(\PdL+14, 38.5);

\fill[cred!40]  (\PdL+2, 24)rectangle(\PdL+5, 26.5);
\node[cgray, font=\scriptsize, anchor=west] at (\PdL+6, 25.2) {$\IA(0)$};
\fill[camb!60]  (\PdL+2, 19.5)rectangle(\PdL+5, 22);
\node[cgray, font=\scriptsize, anchor=west] at (\PdL+6, 20.7) {$\IA(\bar{m})$};
\fill[cblue!80] (\PdL+2, 15)rectangle(\PdL+5, 17.5);
\node[cblue, font=\scriptsize\bfseries, anchor=west] at (\PdL+6, 16.2) {$\MF$};

\node[cnavy, font=\scriptsize, anchor=center]
  at (\PdL+18, 10) {Extends to Other Natural};

\node[cnavy, font=\scriptsize, anchor=center]
  at (\PdL+18, 5) {and Engineered Systems};

\end{tikzpicture}
\caption*{\textbf{Overview of the MF-PID framework.}
  \textbf{(a)~Independent Agents/Samples:} In the baseline, each sample
  evolves under a pre-computed drift; trajectories carry no information
  about one another, yielding no inter-agent coupling.
  \textbf{(b)~Interacting Agents:} MF-PID promotes samples to
  interacting agents whose optimal drift $u^{(*)}_t$ depends
  self-consistently on the evolving population density $\rho_t$ through
  the effective potential $V^{(\mathrm{eff})}_t$.
  \textbf{(c)~LQG:} Two analytically tractable regimes---
  \emph{Linear--Quadratic--Gaussian} (LQG), where the
  infinite-dimensional MF system reduces to Riccati ODEs with
  variance/mean decoupling; and \emph{Gaussian-Mixture} (GM) representation for densities with Piece-Wise-Constant (PWC) in time protocol for $\beta_t$ results in an explicit -- gradient-of-log-of-Gaussian mixture expression for the score function.
  \textbf{(d)~DR:} Application to demand-response control of a
  building TCL fleet; MF coordination achieves 11.6\,\% (Scenario~A)
  and 22.6\,\% (Scenario~B) reductions in cumulative control energy
  $\mathcal{E}$ relative to the unguided IA baseline.}
\label{fig:overview}
\end{figure*}


\section*{Introduction}
\label{sec:intro}

Generative AI has been transformed by diffusion models, which frame sample generation as a stochastic process steered from noise to data \cite{sohl-dickstein_deep_2015,ho_denoising_2020,song_score-based_2021}. A key structural feature of these models --- shared with other generative models, e.g.\ normalizing flows \cite{rezende_variational_2015,papamakarios_normalizing_2021} --- is that samples are generated \emph{independently}: the trajectory of one particle carries no information about any other. Similarly, stochastic optimal transport (SOT) and Schr\"odinger bridge formulations \cite{pavon_free_1991,leonard_survey_2013,chen_optimal_2017} cast distribution matching as an independent-particle path optimization, yielding tractable convolutions of Green functions but discarding inter-particle information; stochastic interpolants \cite{albergo_building_2023} construct flexible transport bridges between arbitrary densities via tunable continuous-time stochastic processes, recovering the Schrödinger bridge as a special limit --- again in an independent-particle framework.

A natural complementary question is whether \emph{coordinated} generation can improve efficiency. In physical and engineered systems, collective behavior routinely outperforms individual action: flocks exploit aerodynamic coupling \cite{brambati_learning_2026}, synchronized HVAC fleets reduce peak demand \cite{callaway_tapping_2009,callaway_achieving_2011,mathieu_arbitraging_2015,beil_frequency_2016}, and robotic swarms leverage formation geometry \cite{borra_optimal_2021,kachar_dynamic_2019}. In each case, coupling through a shared field --- air pressure, a power grid, or relative position --- allows the population to achieve its collective objective at lower individual cost.

We make this intuition precise for generative modeling. Specifically, we extend the Path Integral Diffusion (PID) framework \cite{behjoo_harmonic_2025,chertkov_adaptive_2025,chertkov_generative_2025} --- itself a strict generalization of the Schr\"odinger bridge \cite{pavon_free_1991,leonard_survey_2013, chen_optimal_2017,caluya_reflected_2020,teter_contraction_2023}, recovered as the special case of PID --- to a \emph{Mean-Field} (MF) setting in which an ensemble of agents, each performing a controlled diffusion, interacts through its evolving empirical distribution. In the limit of infinitely many agents this yields a closed McKean--Vlasov stochastic control problem: the optimal drift of each agent depends self-consistently on the population density, introducing a nonlinear coupling absent in Independent-Agent (IA) formulations.

\textbf{Four contributions.}
\emph{First}, we derive the MF-PID equations in both terminal-cost
and SOT (hard marginal constraint) formulations (SI\,\S1), identifying
the precise way in which mean-field coupling breaks global
integrability of the IA case while preserving local analytical
structure.

\emph{Second}, we show that when the base drift is linear, the
interaction potential quadratic, and the target Gaussian, the
infinite-dimensional mean-field system collapses to a finite set of
coupled Riccati and linear ODEs (SI\,\S2).
This \emph{LQG benchmark} is the first closed-form mean-field
generalization of PID, providing explicit energetic comparisons
between MF and IA strategies.
The covariance dynamics are identical in the MF and IA cases; MF
coordination operates entirely through the linear coefficient $s_t$,
giving a clean analytical separation of the energetic benefit.

\emph{Third}, we prove that for a quadratic interaction potential with
zero base drift, the self-consistent MF guidance is the \emph{exact}
linear interpolant between initial and target global means
(Theorem, SI\,\S3.2).
The result is independent of the $\beta$-schedule and the shape of
both distributions.
The proof rests on a structural cancellation in the It\^o--HJB
system: nonlinear score terms cancel exactly in the mean acceleration
equation, leaving a zero-force condition that forces the mean to
evolve linearly.
This converts what appeared to be a nonlinear fixed-point problem into
a one-shot explicit computation.

\emph{Fourth}, we extend analytic tractability to multi-modal
\emph{Gaussian-mixture} targets via a piecewise-constant (PWC)
protocol (SI\,\S3).
With the guidance known analytically, the score function is assembled
in a single pass through closed-form Green-function coefficients
without iteration.
The formulation naturally accommodates non-delta initial distributions,
enabling closed-form transport bridges.

\textbf{Application.}
We demonstrate MF-PID on a physically motivated \emph{Demand Response}
(DR) scenario for large ensembles of multi-zone buildings.
Post-curtailment recovery is cast as a Gaussian-mixture SOT bridge
with $K$ sub-population types and $d$ zones per building.
MF coordination reduces cumulative control energy by 19--24\%
relative to independent-agent baselines while exactly matching the
terminal temperature distribution.
Three scalability properties are verified numerically:
(\emph{i})~the per-zone energy saving is invariant across $d=1$--$32$
zones, confirming the theorem's dimension-independence;
(\emph{ii})~the saving grows consistently (19\%--22\%) as fleet
heterogeneity increases from $K=2$ to $K=8$ sub-types;
(\emph{iii})~adding AR(1) inter-zone thermal coupling ($\rho=0$--$0.8$)
leaves the saving unchanged at $\approx21\%$.
The guidance is a single linearly-interpolated $d$-vector,
broadcast once to the fleet with no iteration required.

\textbf{Relation to existing work.} 

The value of inter-sample communication is not new to filtering. Ensemble Kalman filters (EnKF) \cite{evensen_data_2009} famously replace an intractable Gaussian update with a finite-ensemble approximation in which every particle is corrected by the empirical covariance of the full ensemble — a linear, observation-driven coupling that breaks particle independence at each assimilation step. Our setting differs in three respects: (i) there are no sequential observations; the coupling arises instead from a terminal-cost objective and an explicit interaction potential; (ii) the interaction is nonlinear and self-consistent, governed by a McKean–Vlasov equation rather than a Kalman gain; and (iii) optimality is measured by control energy rather than posterior approximation error. MF-PID can therefore be read as the continuous-time, generative-transport analogue of the EnKF idea: structured inter-sample communication in the service of a well-defined variational objective.

Mean-field Schr\"odinger bridges have been formulated and their ergodic and propagation-of-chaos properties analyzed \cite{backhoff-veraguas_mean_2019,hernandez_propagation_2024}. Our work extends this line by placing mean-field coupling within the Path-Integral Diffusion (PID) framework introduced in \cite{behjoo_harmonic_2025}, which strictly generalizes the classical Schr\"odinger bridge. The independent-agent bridge is recovered as the special case $V_t=f_t=A_t=0$ of PID. Harmonic PID (H-PID) further allows affine base drift and quadratic potentials while preserving Gaussian Green functions \cite{behjoo_harmonic_2025} (see also related analysis in \cite{teter_weyl_2024}), and Guided PID introduces a time-dependent quadratic steering potential, $V_t(x) = \beta_t(x - \eta_t)^2/2$, that shapes trajectories without coupling particles \cite{chertkov_generative_2025}. In all of these constructions, agents remain independent. MF-PID replaces the externally prescribed potential by a self-consistent effective interaction,
\[
V_t^{(\mathrm{eff})}(x)=\frac{1}{2}\int \beta_t (x-y)^2\,p_t(y)\,\mathrm{d}y,
\]
thereby promoting samples to interacting agents governed by a McKean--Vlasov stochastic control problem. Unlike existing mean-field Schr\"odinger formulations \cite{backhoff-veraguas_mean_2019,hernandez_propagation_2024}, our setting incorporates both a nontrivial base drift and an explicit interaction potential. After the Cole--Hopf transform, the system reduces to quasi-linear HJB--KFP equations coupled through the evolving density, yielding a class of mean-field entropic stochastic optimal transport problems not previously analyzed.

While linear--quadratic mean-field control and games are classical \cite{huang_large-population_2007,bensoussan_mean_2013}, they have not been synthesized with entropic SOT or diffusion-based generative modeling. 

Stochastic interpolants \cite{albergo_building_2023} provide a broad independent-particle framework for flows and diffusions -- MF-PID add structure (via PID construct) and then lifts this paradigm to the interacting regime. Finally, the piecewise-constant analytic machinery developed in \cite{chertkov_adaptive_2025,chertkov_generative_2025} is extended here to accommodate mean-field coupling without sacrificing closed-form tractability.

\section*{Results}
\label{sec:results}

\subsection*{The MF-PID Framework}

\textbf{Setup.}
Consider $N$ exchangeable agents, each evolving under the controlled
It\^o diffusion
\begin{equation}
  \dd x_t^{(i)}
  = \bigl(f_t(x_t^{(i)}) + u_t^{(i)}\bigr)\dd t + \dd W_t^{(i)},
  \quad x_0^{(i)} = 0,
  \label{eq:sde-agent}
\end{equation}
where $f_t$ is a base drift, $u_t^{(i)}$ is the control (score),
and $W_t^{(i)}$ are independent standard Brownian motions.
As $N\to\infty$, the marginal density $p_t$ is governed by the MF
Kolmogorov--Fokker--Planck (KFP) equation, and each agent obeys the
McKean--Vlasov SDE
\begin{equation}
  \dd x_t = \bigl(f_t(x_t) + u_t(x_t, p_t)\bigr)\dd t + \dd W_t,
  \quad x_0 = 0.
  \label{eq:mckean}
\end{equation}

\textbf{Optimal control.}
We minimize the mean-field cost
\begin{equation}
  J_t(x) = \inf_u \,\mathbb{E}\!\int_t^1\!\!
    \Bigl(\tfrac{1}{2}\|u_{t'}\|^2
    + \!\int\! V_{t'}(x_{t'}{-}y)\,p_{t'}(y)\,\dd y
    \Bigr)\dd t',
  \label{eq:cost}
\end{equation}
where $V_t(\cdot)$ is an interaction potential penalizing relative
displacement.
Under the Hopf--Cole substitution $J_t = -\log\psi_t$, the HJB
equation for $\psi$ becomes linear, and the optimal control is
$u_t^* = \nabla_x\!\log\psi_t$.
In the SOT formulation, a hard constraint $p_1 = \ptar$ replaces
the terminal penalty, and the optimal control is
\begin{equation}
  \sstar_t(x)
  = \nabla_x \log \int\!\! \ptar(y)
      \frac{G_t^{(-;\mfeff)}(x;y)}{G_1^{(+;\mfeff)}(y;0)}\,\dd y,
  \label{eq:u-star-SOT}
\end{equation}
where $G_t^{(\pm;\mfeff)}$ are Green functions of the HJB/KFP system
evaluated under the effective potential
$\Veff_t(x) = \int V_t(x{-}y)\,p_t(y)\,\dd y$.

\textbf{The MF coupling.}
In the IA case the Green functions can be computed independently of
$p_t$; in the MF case, $G_t^{(\pm;\mfeff)}$ depend on $p_t$ through
$\Veff_t$, and $p_t$ in turn depends on the Green functions.
This self-referential structure is the source of both the difficulty
and the power of MF-PID.
Full derivations are given in SI\,\S1.

\subsection*{LQG Benchmark: Analytic Closed-Form Solution}

\textbf{Model.}
We specialize to a linear base drift $f_t(x) = A_t x + B_t$,
quadratic interaction $V_t(x) = \tfrac{1}{2}x^\top Q_t x$
($Q_t \succeq 0$), and Gaussian target
$\ptar = \calN(m_1^{(\mathrm{tar})},\Sigma_1^{(\mathrm{tar})})$.
Gaussianity is preserved under the controlled dynamics, so
$p_t = \calN(m_t,\Sigma_t)$ for all $t \in [0,1]$.
The effective potential is
$\Veff_t(x) = \tfrac{1}{2}(x{-}m_t)^\top Q_t(x{-}m_t)
+ \tfrac{1}{2}\tr(Q_t\Sigma_t)$
and a quadratic ansatz for $J_t$ yields the closed system
\begin{align}
  -\dot S_t &= Q_t + S_t A_t + A_t^\top S_t - S_t^2,
  \label{eq:riccati-S}\\
  -\dot s_t &= -Q_t m_t + A_t^\top s_t - S_t s_t + S_t B_t,
  \label{eq:riccati-s}\\
  \dot m_t &= (A_t - S_t)m_t + B_t - s_t,
  \label{eq:mean-ode}\\
  \dot\Sigma_t &= (A_t-S_t)\Sigma_t + \Sigma_t(A_t-S_t)^\top + I,
  \label{eq:cov-ode}
\end{align}
with boundary conditions $m_0=0$, $\Sigma_0=0$,
$m_1=m_1^{(\mathrm{tar})}$, $\Sigma_1=\Sigma_1^{(\mathrm{tar})}$.
The optimal MF control is affine: $u_t^*(x) = -S_t x - s_t$.

A key structural observation is that
\eqref{eq:riccati-S}--\eqref{eq:cov-ode} decouple into two
sequential sub-problems: (i) a \emph{variance block}
$(S_t,\Sigma_t)$ that depends only on $Q_t$ and $A_t$ and is
\emph{identical} in the MF and IA cases; and (ii) a \emph{mean block}
$(m_t, s_t)$ whose source term $-Q_t m_t$ couples to the current
population mean rather than a fixed exogenous centre.
MF coordination therefore operates entirely through the linear
coefficient $s_t$ --- a clean analytical separation of the energetic
benefit.

\textbf{Scalar TCL example.}
We illustrate with $d=1$, $f(x) = -\kappa x$
(Ornstein--Uhlenbeck relaxation), and $V(x) = \tfrac{q}{2}x^2$.
This models a population of Thermostatically Controlled Loads (TCLs)
\cite{callaway_tapping_2009,callaway_achieving_2011,hao_aggregate_2015,grammatico_mean_2015,metivier_mean-field_2020,valenzuela_statistical_2023},
where $x_t$ is a normalized temperature deviation.
Let $\Delta = \sqrt{\kappa^2+q}$.
A closed-form shooting procedure (SI\,\S2) determines the unique
$\rho \in (-1,1)$ such that the bridge constraint
$\Sigma_1 = (\sigma^{(\mathrm{tar})})^2$ is satisfied:
\begin{equation}
  \rho = \frac{A-1}{Ar_0-1},
  \quad
  A = \frac{2\Delta(\sigma^{(\mathrm{tar})})^2}{1-r_0},
  \quad
  r_0 = e^{-2\Delta},
  \label{eq:rho}
\end{equation}
after which $S_t$, $\Sigma_t$, and the mean
$m_t = m^{(\mathrm{tar})}\sinh(\kappa t)/\sinh(\kappa)$
are all explicit.

We compare the MF bridge to the IA family parameterised by an
exogenous centre $\bar m \in [0, m^{(\mathrm{tar})}]$.
Since $S_t$ and $\Sigma_t$ are identical across schemes, the
comparison reduces to the linear coefficient $s_t$.
Three performance metrics are used: the mean trajectory $m_t$,
the instantaneous control power
$\calP(t) = S_t^2\sigma_t^2 + (S_t m_t + s_t)^2$,
and the cumulative energy $\calE(t) = \int_0^t \calP(u)\dd u$.

\begin{figure}[h!]
  \centering
  \includegraphics[width=0.48\linewidth]{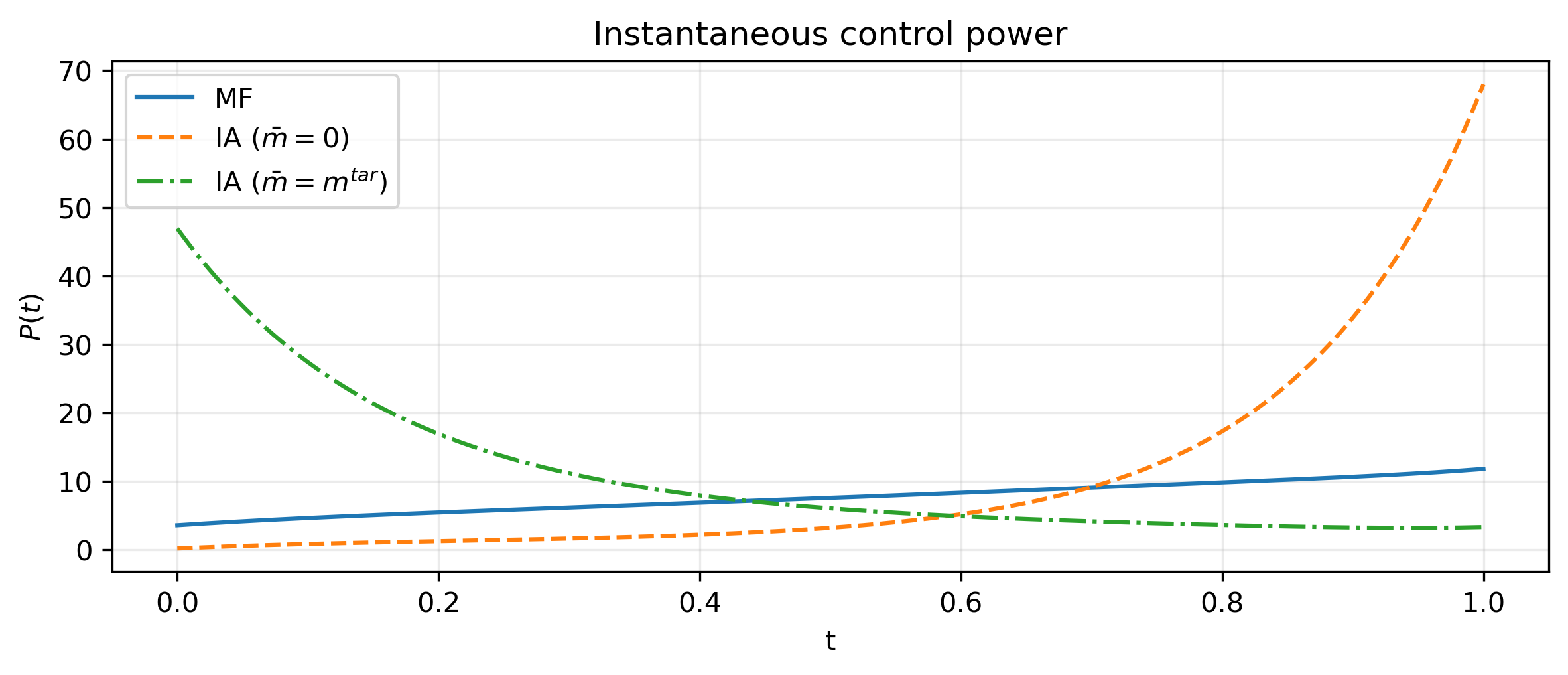}
  \includegraphics[width=0.48\linewidth]{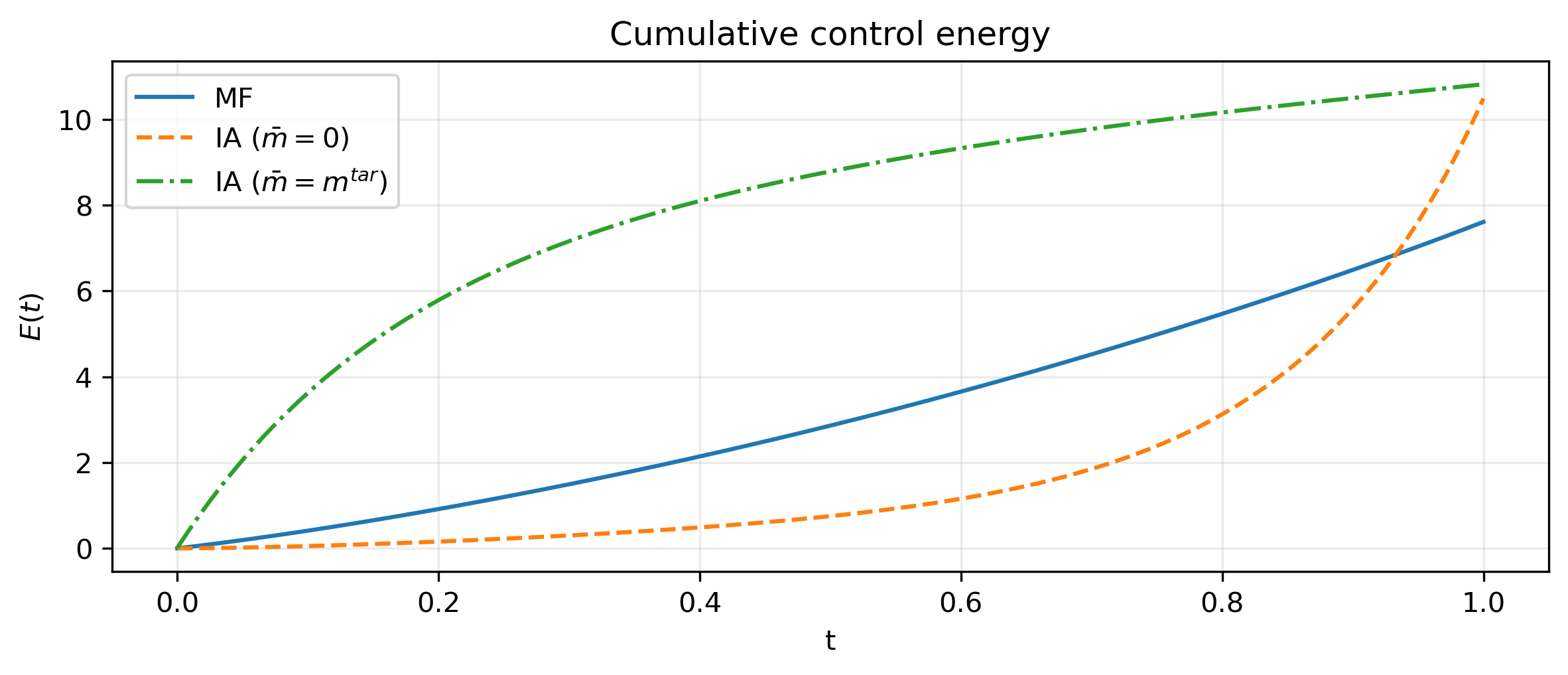}
  \includegraphics[width=0.48\linewidth]{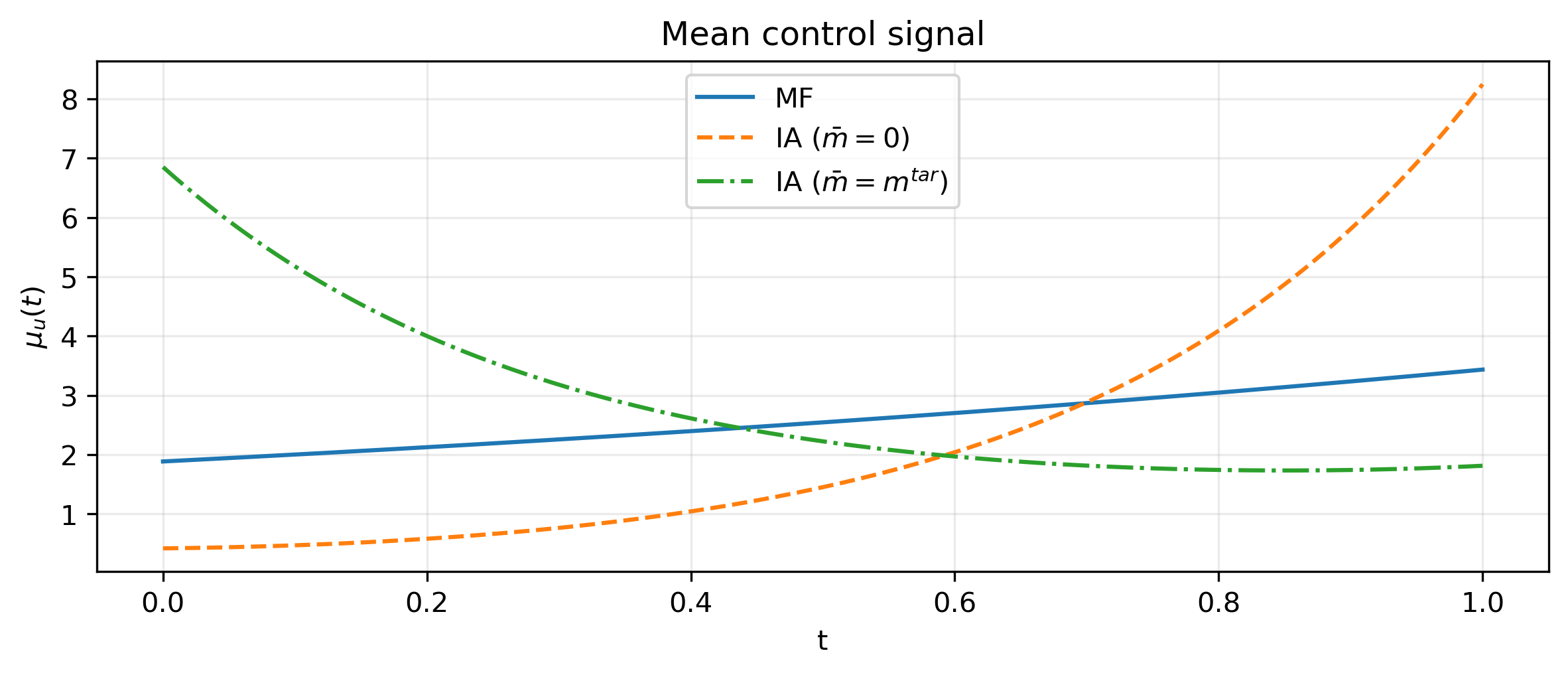}
  \includegraphics[width=0.46\linewidth]{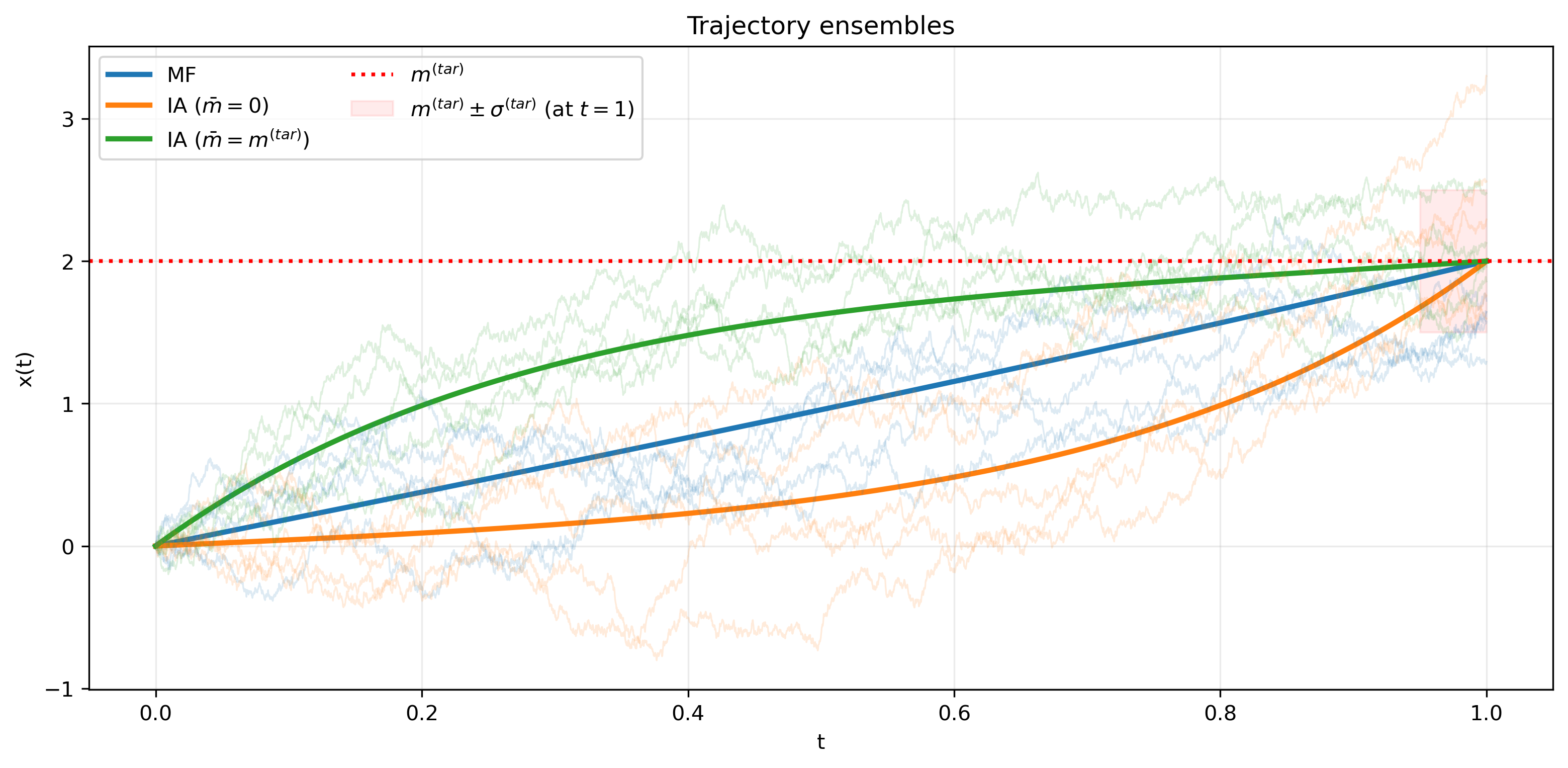}
  \caption{\textbf{MF vs.\ IA controls in the scalar TCL example.}
    \emph{Top left:} Trajectory ensembles for MF (blue),
    IA with $\bar{m}=0$ (orange), and IA with $\bar{m}=m^{(\mathrm{tar})}$
    (green); thin curves: sample paths, thick: analytic means,
    shaded band: target variance.
    \emph{Top right:} Instantaneous control power $\mathcal{P}(t)$.
    \emph{Bottom left:} Cumulative control energy
    $\mathcal{E}(t)=\int_0^t\mathcal{P}(u)\,\mathrm{d}u$.
    \emph{Bottom right:} Mean trajectories $m_t$.
    All curves are obtained analytically.}
  \label{fig:Gauss-to-Gauss}
\end{figure}

Fig.~\ref{fig:Gauss-to-Gauss} shows that the MF mean follows a smooth
hyperbolic sine arc, while IA means exhibit stronger curvature driven
by $\Delta$.
The MF cumulative energy curve lies \emph{strictly below} both IA
curves throughout $[0,1]$: MF coordination achieves the same terminal
accuracy with reduced total control effort.
The advantage originates entirely from $s_t$, which in the MF case
adapts to the endogenous population mean rather than a fixed exogenous
reference.

\subsection*{Exact Linear MF Guidance}

For the isotropic quadratic potential
$V_t(x) = \tfrac{\beta_t}{2}\|x\|^2$ with $f_t \doteq  0$, the
effective potential reduces to
$\Veff_t(x) = \tfrac{\beta_t}{2}\|x - \nu_t\|^2 + \text{const}$,
where $\nu_t = m_t = \mathbb{E}_{p_t^*}[x]$.
MF-PID therefore becomes a \emph{guided H-PID} whose guidance is
the solution of the self-consistency condition
\begin{equation}
  \nu_t^{(\MF)}
  = \mathbb{E}_{x\sim p_t^{*}(\,\cdot\,|\,\nu^{(\MF)})}[x],
  \quad t\in[0,1].
  \label{eq:mf-consistency}
\end{equation}

\begin{tcolorbox}[colback=gray!7,colframe=gray!45,
  title={\small\bfseries \color{black} Theorem (SI\,Thm.\,3.1)}]
\small
Let $f_t \doteq  0$, $V_t(x)=\tfrac{\beta_t}{2}\|x\|^2$ for any
$\beta_t>0$, and let $\pin$, $\ptar$ be probability measures on
$\Rbb^d$ with finite first moments
$\bar m^{(\mathrm{in})}=\mathbb{E}_{\pin}[x]$,
$\bar m^{(\mathrm{tar})}=\mathbb{E}_{\ptar}[x]$.
Then the self-consistent MF guidance satisfies
\begin{equation}
  \nu_t^{(\MF)} = m_t
  = (1-t)\,\bar m^{(\mathrm{in})} + t\,\bar m^{(\mathrm{tar})}
  \quad\forall\,t\in[0,1],
  \label{eq:nu-linear}
\end{equation}
\emph{exactly}, independently of $\beta_t$, the number and geometry
of mixture components, and the shape of $\pin$ and $\ptar$.
\end{tcolorbox}

\textbf{Proof sketch.}
Applying It\^o's formula to $u_t^*(x_t)$ and differentiating
$\dot m_t = \mathbb{E}[u_t^*(x_t)]$ gives
$\ddot m_t
= \mathbb{E}[\partial_t u_t^* + (u_t^*\cdot\nabla)u_t^*
+ \tfrac{1}{2}\Delta u_t^*]$.
Taking the spatial gradient of the HJB equation yields
$\partial_t u_t^*
= \nabla_x\Veff_t - (u_t^*\cdot\nabla)u_t^* - \tfrac{1}{2}\Delta u_t^*$,
so the nonlinear score terms cancel \emph{exactly} inside
$\ddot m_t$, leaving
$\ddot m_t = \mathbb{E}[\nabla_x\Veff_t(x_t)]
= \beta_t(m_t - \nu_t^{(\MF)})$.
The self-consistency condition $\nu_t^{(\MF)}=m_t$ then forces
$\ddot m_t = 0$, and the linear arc
\eqref{eq:nu-linear} follows from the boundary conditions
$m_0=\bar m^{(\mathrm{in})}$, $m_1=\bar m^{(\mathrm{tar})}$.
Full proof in SI\,\S3.2.

\textbf{Corollaries.}
When $\pin = \delta(\cdot)$, equation~\eqref{eq:nu-linear} reduces to
$\nu_t^{(\MF)} = t\,\bar m^{(\mathrm{tar})}$ (SI\,Cor.\,3.2).
For $f_t(x) = -\kappa x$ the same It\^o--HJB cancellation yields
$\ddot m_t = \kappa^2 m_t$, recovering the $\sinh$-arc
$m_t = m^{(\mathrm{tar})}\sinh(\kappa t)/\sinh(\kappa)$
of the LQG benchmark (SI\,Cor.\,3.3).

\begin{tcolorbox}[colback=blue!4,colframe=blue!45,
  title={\small\bfseries  Perspective (alternative viewpoint)}]
\emph{Self-consistent guidance.}
Although we derive the linear guidance by formulating MF-PID and reducing it to guided H-PID, the logic can be reversed.
Start from \emph{guided} PID with prescribed $\nu_t$ and potential
$V_t(x)=\frac{\beta_t}{2}\|x-\nu_t\|^2$.
A natural ``closure'' is to choose $\nu_t$ so that it matches the mean it induces:
\[
\nu_t=\mathbb{E}_{p_t^{(\nu)}}[x],\qquad t\in[0,1],
\]
where $p_t^{(\nu)}$ is the time-marginal under guidance $\nu$.
This turns guidance design into a fixed-point problem in trajectory space.
Our main theorem shows that for $f_t\doteq 0$ and quadratic interaction this fixed point is explicit:
\[
\nu_t=(1-t)\,\bar m^{(\mathrm{in})}+t\,\bar m^{(\mathrm{tar})},
\]
independent of $\beta_t$ and of the shapes of the endpoint distributions.
In this view, MF-PID provides the variational/self-consistency principle that \emph{selects} $\nu_t$,
while guided PID supplies the operational generative mechanism.
\end{tcolorbox}

\textbf{Practical implication.} Eq.~\eqref{eq:nu-linear} provides the MF guidance \emph{in closed form, without any iteration}: one computes the two global means $\bar m^{(\mathrm{in})}$ and $\bar m^{(\mathrm{tar})}$, sets $\nu_t^{(\MF)} = (1-t)\bar m^{(\mathrm{in})} + t\bar m^{(\mathrm{tar})}$, and proceeds directly to score function evaluation via the closed-form Green-function coefficients (SI\,\S3.3--3.4). The self-consistency fixed-point iteration is therefore unnecessary for $f_t\doteq  0$ and is retained in SI only for reference and for the $f_t\not\doteq  0$ setting.

\subsection*{Gaussian-Mixture Score and Demand Response}

\textbf{Explicit score assembly.} For a Gaussian-mixture target $\ptar = \sum_{k=1}^K \pi_k\calN(x;\,m_k,\Sigma_k)$ with guidance \eqref{eq:nu-linear} set analytically, the PWC Green-function coefficients within each interval $[t_i, t_{i+1})$ are determined by the Riccati ODEs of SI\,\S3.3. With the guidance pre-computed, all $K$ sets of coefficients are evaluated in a single forward pass (no outer iteration). The score function then takes the closed form
\begin{equation}
  \sstar_t(x)
  = b_t^{(-)}\bigl(\hat y(t;x) - \Upsilon_t(x)\bigr),
  \label{eq:score}
\end{equation}
where $\hat y(t;x)$ is the mixture-weighted posterior target-component mean and $\Upsilon_t(x)$ an affine function of $x$; all quantities are explicit hyperbolic functions of the PWC coefficients (SI\,\S3.3--3.4, eqs.~S3.4--S3.9). No neural network or iterative solve is required.

\textbf{Demand-response setting.} We consider a fleet of buildings, each with $d$ thermal zones, partitioned into occupied ($\pi_\mathrm{occ}=0.60$, $m^{(\mathrm{tar})}_\mathrm{occ}=0.0$, $\sigma^{(\mathrm{tar})}_\mathrm{occ}=0.20$) and unoccupied ($\pi_\mathrm{unocc}=0.40$, $m^{(\mathrm{tar})}_\mathrm{unocc}=1.5$, $\sigma^{(\mathrm{tar})}_\mathrm{unocc}=0.30$) sub-populations (non-dimensional units; $0\,\widehat{=}\,20\,^\circ\mathrm{C}$, unit~$\widehat{=}\,3\,^\circ\mathrm{C}$). A curtailment event displaces temperatures from setpoints; the \emph{recovery} phase is cast as a GM-to-GM H-PID bridge. By Theorem~\eqref{eq:nu-linear} the guidance $\nu_t^{(\MF)}=(1{-}t)\bar m^{(\mathrm{in})}+t\bar m^{(\mathrm{tar})}$ is exact, requiring no iteration. We compare three strategies --- IA($\nu{=}0$), IA($\nu{=}\bar{m}$), and MF~\eqref{eq:nu-linear} --- under two post-curtailment scenarios (Fig.~\ref{fig:DR-trajectories}; full diagnostics in SI\,\S4):
\begin{itemize}
  \item \textbf{Scenario~A (wide):} modes overlap heavily ($\sigma^{(\mathrm{in})}=3.0$), aggressive curtailment.
  
  \item \textbf{Scenario~B (narrow):} modes well-separated ($\sigma^{(\mathrm{in})}=(0.5,0.7)$), mild curtailment.
\end{itemize}

\begin{figure}[h!]
  \centering
  \includegraphics[width=\linewidth]{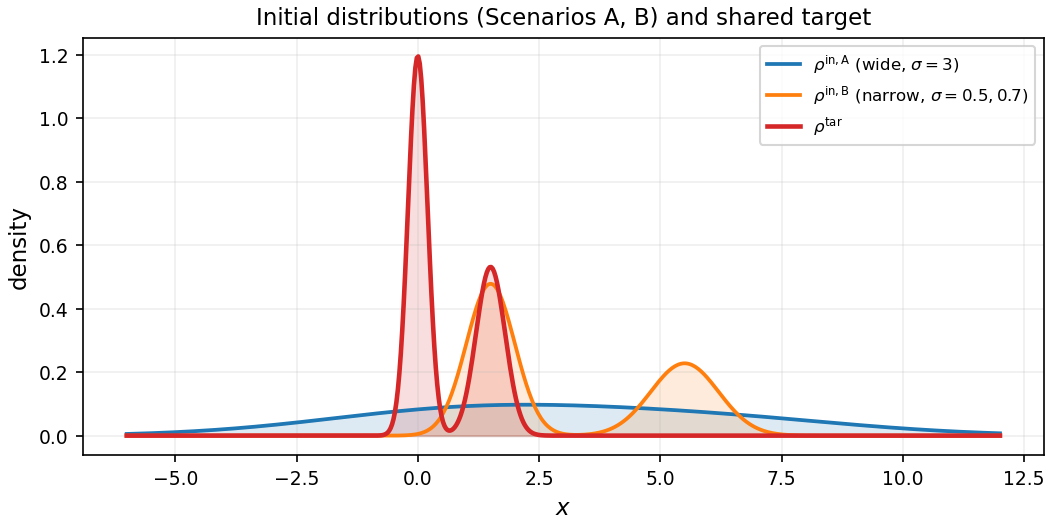}
  \includegraphics[width=\linewidth]{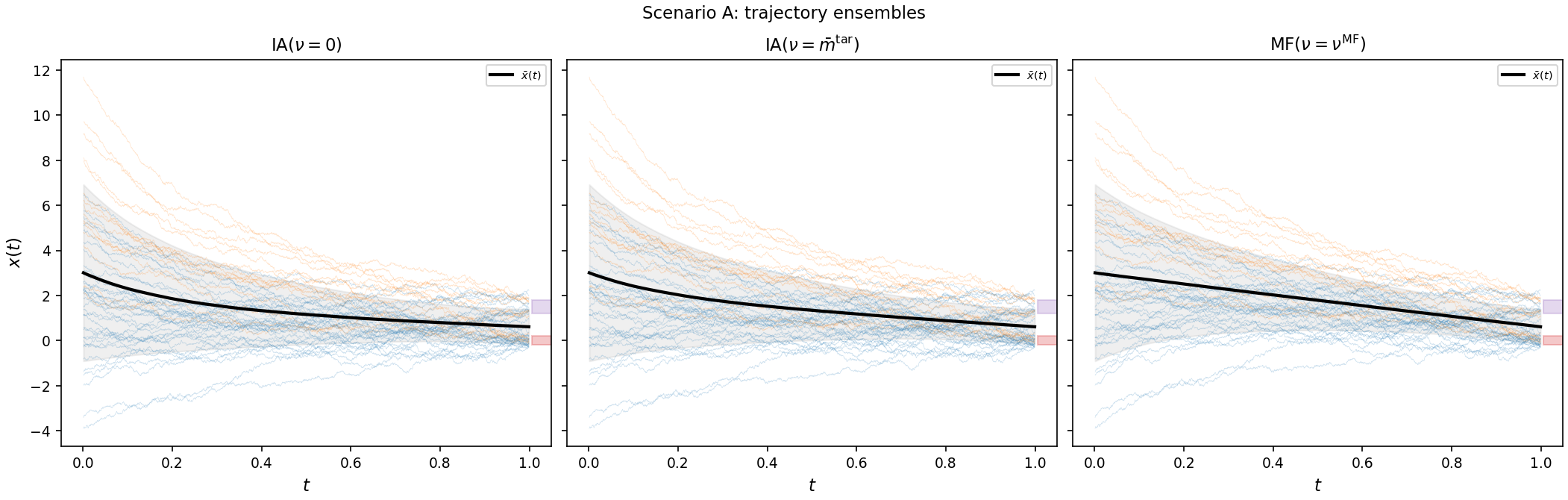}
  \includegraphics[width=\linewidth]{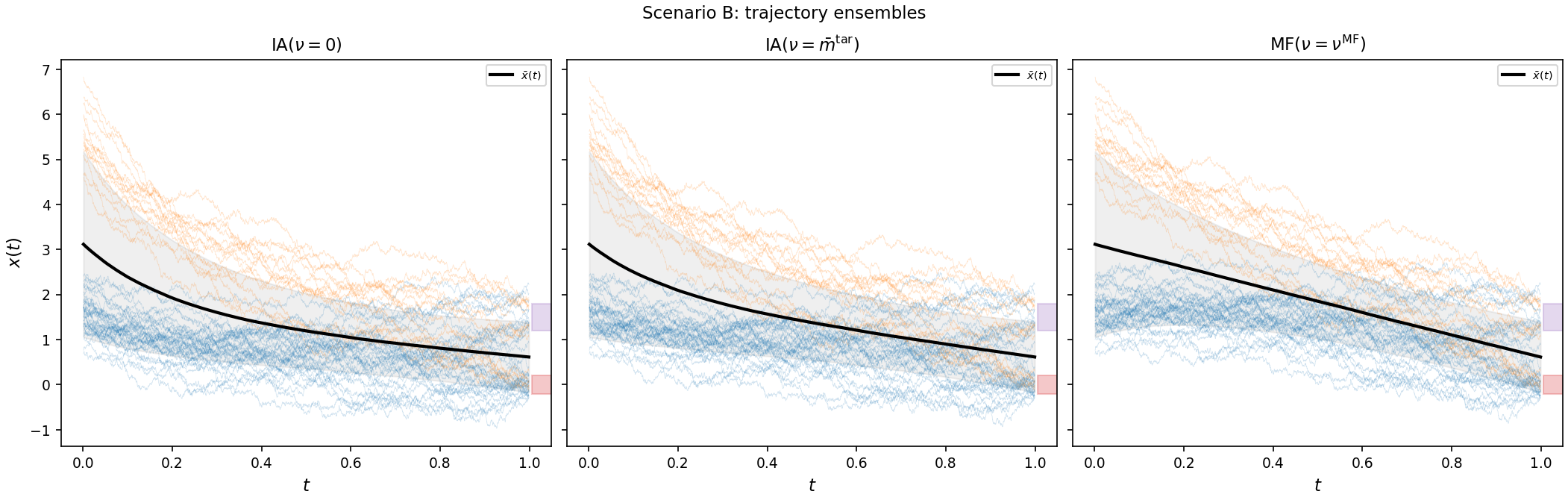}
  \caption{\textbf{Top:} Initial and target distributions.
    \textbf{Middle/Bottom:} Sample trajectories (50 paths) for Scenarios~A (middle) and B (bottom). Blue: occupied; orange: unoccupied. Black: ensemble mean; gray band: $\pm1\sigma$; shaded rectangles: target $\pm\sigma_k^{(\mathrm{tar})}$ bands.}
  \label{fig:DR-trajectories}
\end{figure}

\textbf{Energy results.} The ordering $\calE^{\MF}<\calE^{\IA(\bar m)}<\calE^{\IA(0)}$ holds in both scenarios. For Scenario~A: $\calE^{\MF}=27.67$, $\calE^{\IA(\bar m)}=29.68$, $\calE^{\IA(0)}=31.30$ (\textbf{11.6\%} saving). For Scenario~B: $\calE^{\MF}=13.27$, $\calE^{\IA(\bar m)}=15.47$, $\calE^{\IA(0)}=17.15$ (\textbf{22.6\%} saving). Table~\ref{tab:energy} decomposes these totals by sub-population: MF slightly increases cost for the easy (occupied) mode while substantially reducing it for the expensive (unoccupied) one --- in Scenario~B the unoccupied energy drops from 37.40 to 28.07 ($-$25\%), invisible at the level of the global mean.

\begin{table}[h!]
\centering
\caption{\textbf{Per-mode control energy.} MF guidance \eqref{eq:nu-linear} redistributes effort from the occupied to the unoccupied sub-population, yielding net savings.}
\label{tab:energy}
\small
\setlength{\tabcolsep}{4pt}
\begin{tabular}{l ccc ccc}
\toprule
& \multicolumn{3}{c}{\textbf{Scenario A}}
& \multicolumn{3}{c}{\textbf{Scenario B}} \\
\cmidrule(lr){2-4}\cmidrule(lr){5-7}
Method & $\calE_\mathrm{occ}$ & $\calE_\mathrm{unocc}$ & Total
       & $\calE_\mathrm{occ}$ & $\calE_\mathrm{unocc}$ & Total \\
\midrule
$\nu=0$         & 13.89 & 56.92 & 31.30 &  3.38 & 37.40 & 17.15 \\
$\nu=\bar m$    & 13.43 & 53.59 & 29.68 &  2.63 & 34.36 & 15.47 \\
$\nu=\nu^{\MF}$ & 14.72 & 46.73 & \textbf{27.67}
                &  3.21 & 28.07 & \textbf{13.27} \\
\bottomrule
\end{tabular}
\end{table}

\textbf{Mechanism.} The MF advantage amplifies from Scenario~A to B because fleet heterogeneity increases. In Scenario~A the two modes overlap heavily, the ensemble behaves nearly as a single Gaussian, and a fixed constant guidance already captures most of the benefit. In Scenario~B the modes are quasi-independent: the unoccupied cluster travels ${\approx}4$ units while the occupied cluster moves only ${\approx}1.5$. The self-consistent guidance~\eqref{eq:nu-linear} adapts to this asymmetry, whereas any constant $\nu$ cannot. The energy savings arise from the \emph{timing} of guidance imposed by the $\beta_t$ schedule on the exact linear trajectory, not from a nonlinear displacement of $\nu_t$ itself. Mean-trajectory, power-spectrum, and guidance residual diagnostics for both scenarios are provided in SI\,\S4.

\subsection*{Multi-Zone Scalability}

We now lift the $d=1$ restriction and test the method across three axes: zone count $d$, fleet heterogeneity $K$, and inter-zone thermal coupling. In all cases the MF guidance remains \eqref{eq:nu-linear}: a single linearly-interpolated $d$-vector requiring no iteration.

\textbf{Dimension sweep ($d=1$--$32$, $K=2$).} Each particle represents a building whose state $x\in\mathbb{R}^d$ encodes temperature deviations of $d$ zones (perimeter and interior, distinguished by a sinusoidal zone-type vector). We use Scenario-B parameters broadcast isotropically to $d$~dimensions. Results are summarized in Table~\ref{tab:d_sweep}.

\begin{table}[h!]
\centering
\caption{\textbf{Dimension sweep} ($K=2$, Scenario~B parameters). Per-zone energy $\calE(1)/d$ is invariant across $d=1$--$32$, confirming Theorem~\eqref{eq:nu-linear}'s dimension-independence. Wall-clock time (single CPU, $B=4{,}000$ particles, $n_\mathrm{steps}=2{,}500$) scales sub-cubically in this range; the asymptotic $\mathcal{O}(d^3)$ Cholesky cost dominates beyond $d\gg 32$.}
\label{tab:d_sweep}
\small
\setlength{\tabcolsep}{5pt}
\begin{tabular}{r ccc c r}
\toprule
$d$ & $\calE/d$ (MF) & $\calE/d$ (IA$_0$) & $\calE/d$ (IA$_m$) & Saving & Time (s) \\
\midrule
 1 & 12.26 & 16.17 & 14.34 & 24.2\% &  88 \\
 2 & 13.07 & 16.96 & 15.13 & 23.0\% &  93 \\
 4 & 13.37 & 17.22 & 15.41 & 22.4\% & 105 \\
 8 & 13.57 & 17.40 & 15.59 & 22.0\% & 106 \\
16 & 13.50 & 17.32 & 15.51 & 22.1\% & 152 \\
32 & 13.57 & 17.19 & 15.42 & 21.1\% & 353 \\
\bottomrule
\end{tabular}
\end{table}

The key finding is that $\calE(1)/d \approx 13.4\pm0.2$ (MF) and $17.2\pm0.3$ (IA$_0$) across the entire range: \emph{the per-zone energy saving of $\approx22\%$ is independent of $d$}. Each additional zone costs the same to coordinate as the first. This directly validates Theorem~\eqref{eq:nu-linear}: the guidance adds $\mathcal{O}(d)$ compute (a single mean vector) regardless of how many zones share the building.

The wall-clock time grows from 88\,s ($d=1$) to 353\,s ($d=32$) --- a 4$\times$ increase for a 32$\times$ increase in dimension. For $d\leq 8$ the overhead from the Python time-stepping loop dominates the $d\times d$ Cholesky factorizations; the compute cost scales roughly as $\mathcal{O}(d^{1.3})$ over $d=1$--$32$ and is expected to transition to the asymptotic $\mathcal{O}(d^3)$ regime for $d\gg 32$. For real-time building control ($d\leq 50$, dispatch intervals $\geq\!1$\,min), this places MF-PID comfortably within operational time budgets.

\textbf{Fleet heterogeneity sweep ($K=2$--$8$, $d=4$).} We fix $d=4$ and vary the number of building sub-types $K$ from 2 to 8. Component means are spaced uniformly and initial distributions are displaced by +4 units (aggressive curtailment), with weights decreasing geometrically so that \emph{efficient-to-move} buildings are more numerous. The MF saving grows consistently: 19.3\% ($K=2$), 21.0\% ($K=3$), 21.6\% ($K=4$), 22.4\% ($K=8$). The trend is real but moderate because this design keeps the \emph{per-component displacement} constant ($|\Delta m_k|=4$ for all $k$): only the weight asymmetry increases with $K$. The more dramatic contrast (11.6\% vs 22.6\%) seen in the $d=1$ Scenarios~A/B arises from \emph{geometric} mode heterogeneity (different travel distances), which is the dominant driver of MF advantage.

A structural observation from the $K$-sweep: by construction the target global mean $\bar m^{(\mathrm{tar})}=0$ for all $K$, so the two IA baselines ($\nu\doteq 0$ and $\nu\doteq\bar m^{(\mathrm{tar})}$) coincide exactly. This confirms a general property: for any two constant-guidance IA strategies that share the same time-invariant reference, the costs are identical regardless of $K$ or $d$. The MF guidance, by contrast, uses a \emph{time-varying} interpolant and achieves strictly lower cost in all cases.

\textbf{Inter-zone coupling ($\rho = 0$--$0.8$, $d=8$, $K=2$).} Real buildings have spatially correlated zone temperatures through shared walls and HVAC ducts. We model this with a \emph{spatial} AR(1) covariance:
\begin{equation}
  \Sigma^{(k)}_{ij} = \sigma_k^2\,\rho^{|i-j|},
  \quad i,j = 0,\ldots,d{-}1,
  \label{eq:ar1-cov}
\end{equation}
where $|i{-}j|$ is the integer distance between zone indices and $\rho\in[0,1)$ controls how rapidly correlation decays with spatial separation. ``AR(1)'' refers to the first-order autoregressive structure in the \emph{spatial} index: each zone's temperature is most similar to its immediate neighbours, with influence decreasing geometrically with distance. Here $\rho=0$ recovers independent zones, $\rho=0.5$ gives moderate coupling (adjacent-zone correlation 0.5), and $\rho=0.8$ represents strongly coupled perimeter-to-interior heat transfer (correlation 0.64 across two zones, 0.51 across three). (Note -- as a remark towards future work -- that the spatial AR(1) is distinct from accounting for a \emph{spatio-temporal} correlation, which would additionally link the same zone across time steps. Within our framework, temporal dependence between successive states is naturally introduced by a nonzero base drift $f_t(x)$ --- for example, the OU relaxation $f_t(x)=-\kappa x$ of the LQG section couples the current state to its past through the mean-reversion dynamics. Theorem~\eqref{eq:nu-linear} applies with spatial AR(1) covariances unchanged (it requires only finite first moments, not diagonal covariances).)
Results: the MF saving is $22.0\%$ at $\rho=0$, $21.8\%$ at
$\rho=0.5$, and $21.0\%$ at $\rho=0.8$ --- \emph{insensitive
to zone coupling}.
Increasing $\rho$ raises absolute energies (coupled zones
collectively require more coordinated actuation) but the
relative MF advantage is preserved, confirming the method
needs no retuning for realistic building envelopes.

\begin{figure}[h!]
  \centering
  \includegraphics[width=\linewidth]{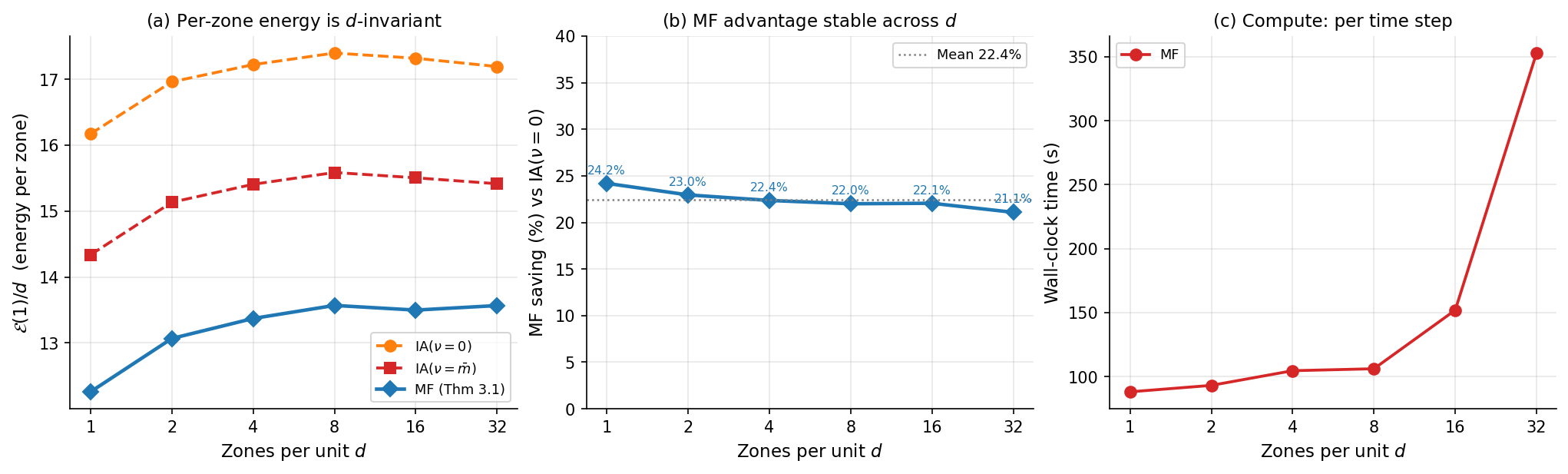}
  \caption{\textbf{Multi-zone scalability ($K=2$).}
    \emph{(a)} Per-zone cumulative energy $\calE(1)/d$ vs.\ number
    of zones $d$.  MF (blue diamonds) is flat at $\approx13.5$;
    IA($\nu{=}0$) (orange) at $\approx17.2$.
    \emph{(b)} MF energy saving (\%) vs.\ $d$: stable at 21--24\%
    across the full range.
    \emph{(c)} Wall-clock time on a single CPU core.
    The dashed line shows the asymptotic $\mathcal{O}(d^3)$ trend
    (Cholesky of $d\times d$ covariance matrices per time step);
    for $d\leq 32$ overhead dominates and scaling is milder.}
  \label{fig:d_sweep}
\end{figure}

\begin{figure}[h!]
  \centering
  \includegraphics[width=0.55\linewidth]{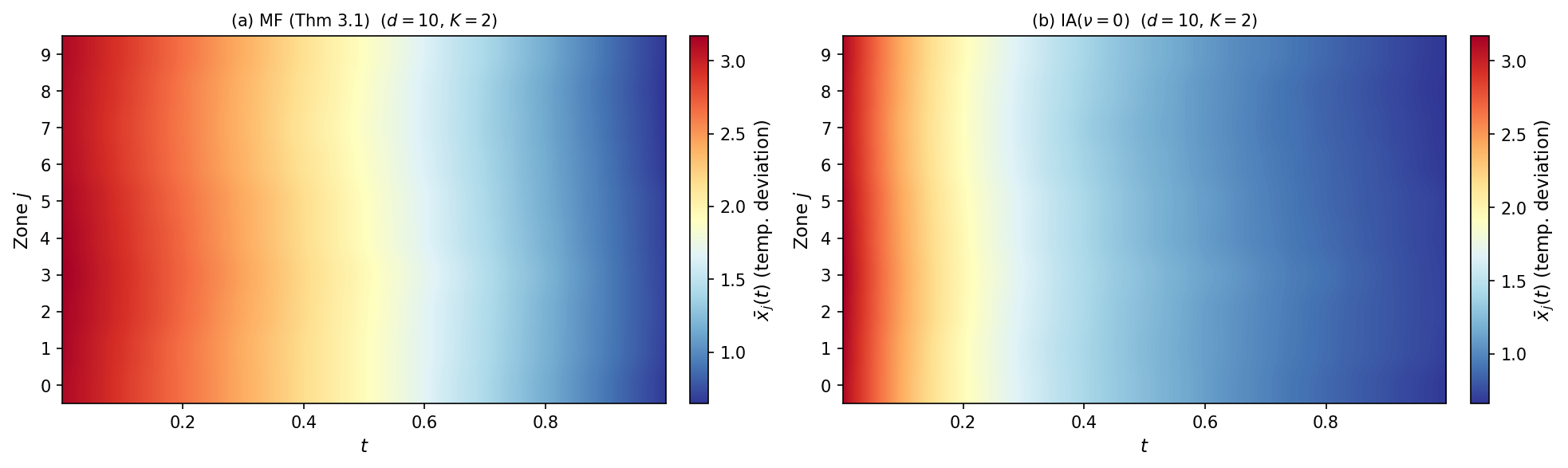}
  \hfill
  \includegraphics[width=0.42\linewidth]{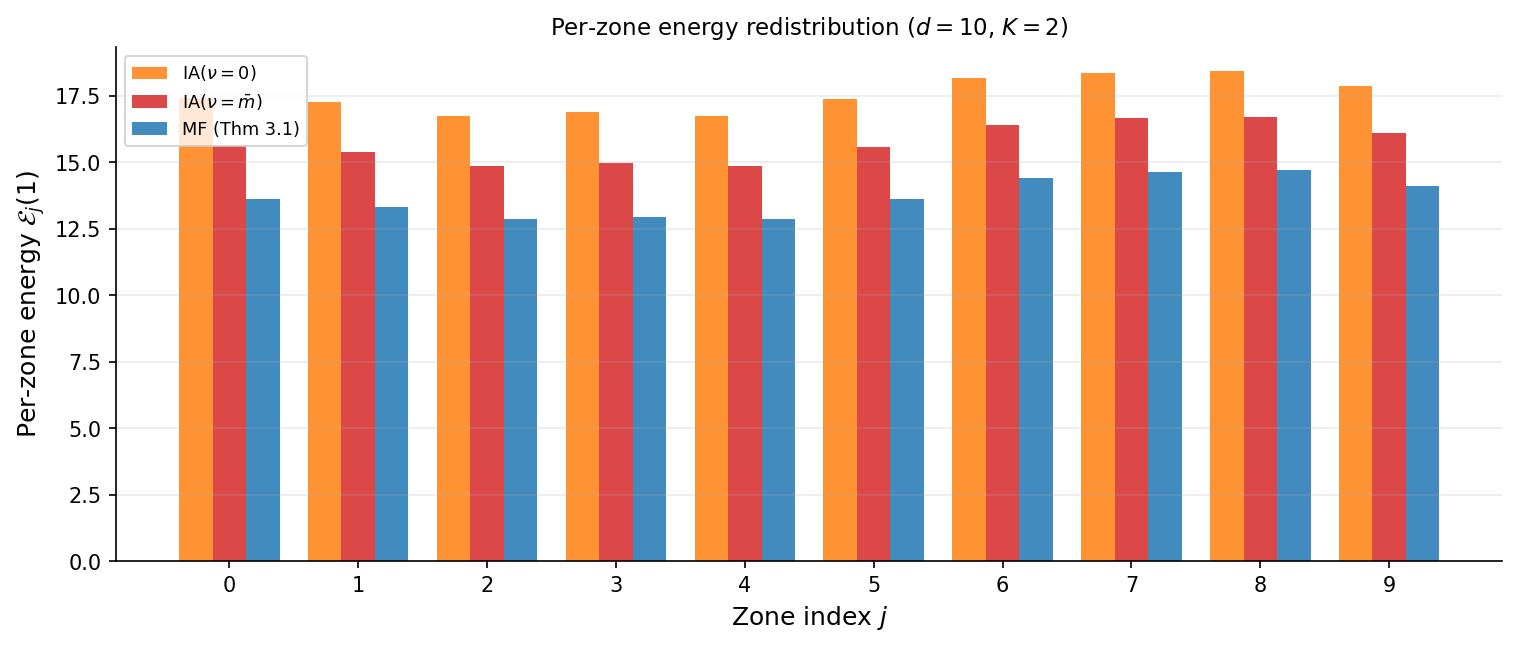}
  \caption{\textbf{Spatial mechanism at $d=10$, $K=2$.}
    \emph{Left:} Zone$\times$time heatmap of ensemble mean temperature
    $\bar x_j(t)$ for MF (top) and IA($\nu{=}0$) (bottom).
    Under MF each zone mean follows a near-linear arc; under IA
    high-displacement zones (top/bottom rows) show stronger curvature.
    \emph{Right:} Per-zone cumulative energy.
    MF redistributes effort: high-displacement zones (e.g.\ zone 5)
    see the largest reduction ($\approx14\%$) while low-displacement
    zones change negligibly.
    Total MF saving: 22.3\%.}
  \label{fig:zone_heatmap}
\end{figure}

\section*{Discussion}
\label{sec:disc}

\textbf{Samples as agents.} MF-PID reveals a precise mathematical duality between generative modeling and coordinated stochastic control: what appears in generative AI as ``sampling from a target distribution'' is equivalent, in the MF limit, to ``redistributing a population under minimal actuation.'' The governing equations --- HJB, KFP, and their Green-function kernels --- do not distinguish between the two interpretations; only the context does. 

The duality between samples and interacting agents also has a distinguished antecedent in data assimilation: ensemble Kalman filters \cite{evensen_data_2009} achieve tractable Bayesian updates by having each particle absorb information from the empirical ensemble covariance. MF-PID extends this philosophy beyond the linear-Gaussian, observation-driven setting: the coupling is encoded in a potential, the interaction is self-consistent across continuous time, and the figure of merit is control energy rather than filtering accuracy.

\textbf{Energetic value of coordination.} The LQG analysis provides a clean analytical proof that MF coordination is energetically superior to exogenously centred IA strategies. The advantage originates in the endogenous population feedback embedded in $s_t$: the MF controller adapts its reference to the current ensemble mean rather than a pre-fixed value. In the Gaussian-mixture setting, the self-consistent field allocates actuation effort proportionally to per-mode transport difficulty.

\textbf{Exact linearity and its implications.} The theorem \eqref{eq:nu-linear} is the central new analytical result of this paper. What appeared empirically as ``near-linearity'' of the guidance is in fact \emph{exact} linearity --- a consequence of a structural cancellation in the It\^o--HJB system that holds for any score function, any distribution, and any $\beta$-schedule. This has three direct practical payoffs. First, the guidance is known analytically before any simulation is run, eliminating online iteration and reducing the communication overhead to a single broadcast of a linearly-interpolated $d$-vector --- ideal for real-time dispatch of large building fleets. Second, the exact linear guidance~\eqref{eq:nu-linear} provides a provably correct initializer for any hybrid neural--analytic extension of MF-PID, removing the need for warm-start iteration in high dimensions. Third, for the OU base drift $f_t(x)=-\kappa x$, the same It\^o--HJB argument yields the $\sinh$-arc in closed form (Corollary, SI\,\S3.2), so the LQG benchmark mean-dynamics result is a special case of a broader structural principle.

\textbf{Open-loop vs.\ closed-loop inference.} The MF-PID construction is \emph{distributionally closed-loop at design time} but \emph{open-loop at inference time}: once $\nu_t^{(\MF)}$ is fixed by~\eqref{eq:nu-linear}, individual samples evolve independently under a pre-determined drift. A natural extension is to replace $\nu_t^{(\MF)}$ in the drift with the empirical ensemble mean $\hat m_t^{(N)} = N^{-1}\sum_{i=1}^N x_t^{(i)}$ during simulation, similarly to the DR implementation in \cite{metivier_mean-field_2020}. Since Theorem \eqref{eq:nu-linear} guarantees that the theoretical guidance equals the population mean, this closed-loop implementation introduces no additional design cost: agents broadcast their current mean and apply the linear schedule. In continuous time this corresponds to a McKean--Vlasov SDE; stability is suggested by the contraction properties established for Schr\"odinger bridge iterations in linear settings \cite{teter_contraction_2023}.

\textbf{Neuralization and high dimensions.} The present construction is deliberately free of neural networks: the analytic GM-to-GM setting serves as a mathematically controlled backbone. For high-dimensional targets beyond the Gaussian-mixture regime, the PWC inner-solve structure provides a natural scaffold for hybridization: learn score corrections beyond quadratic structure using neural networks while retaining analytic Green-function evolution within each PWC interval. Crucially, the exact linear guidance \eqref{eq:nu-linear} provides a provably correct initialization for any such learning procedure. Neural SDE frameworks \cite{tzen_theoretical_2019} provide the theoretical underpinning.

\textbf{Physical time and engineering implications.} Unlike diffusion models operating in synthetic noise time, MF-PID operates in real physical time: recovery horizons, maneuvering windows, or reliability intervals. Control energy, peak actuation, and transient risk are first-class design quantities. The DR application demonstrates three properties of direct relevance to energy-system practitioners. \emph{Dimension invariance}: the per-zone saving of $\approx22\%$ is flat across $d=1$--$32$ zones per building. A dispatcher can scale from a scalar TCL model to a multi-zone office floor without retuning, and the guidance computation adds only $\mathcal{O}(d)$ work (one matrix-vector product per dispatch interval). \emph{Fleet heterogeneity}: the saving grows consistently as the number of distinct building sub-types $K$ (number of components in Gaussian mixture) increases, because MF coordination naturally exploits the disparity in per-type transport distances. \emph{Zone coupling}: AR(1) inter-zone thermal correlation ($\rho=0$--$0.8$) leaves the saving unchanged, validating the approach for realistic building envelopes without any algorithmic modification. Taken together, these findings show that the $\mathcal{O}(d)$ guidance broadcast is not merely a theoretical convenience but a practical enabler of real-time fleet control at building-relevant scales ($d\leq 50$, dispatch intervals $\geq\!1$\,min).

\textbf{Broader outlook.} Three physics-deep research directions emerge. First, nonlinear MF coupling may induce \emph{collective phenomena}: symmetry breaking, mode locking, or phase-transition-like dynamics in multi-modal targets, well-studied in statistical physics but unexplored in generative modeling. Second, MF-PID aligns naturally with the emerging framework of ``sampling decisions'' \cite{chertkov_sampling_2025}, which combines diffusion, transformer, and reinforcement learning under a unified stochastic control umbrella (see also last chapter of \cite{chertkov_mathematics_2025}). Third, the identification of samples with agents positions MF-PID as a bridge between passive data synthesis and active coordination of engineered systems under real-time constraints.

\section*{Methods}
\label{sec:methods}

\textbf{Theoretical framework.} Full derivations of the MF-PID equations in terminal-cost and SOT formulations, the LQG closed-form reduction, and the proof of Theorem~\eqref{eq:nu-linear} are provided in SI\,\S1--\S3. The Hopf--Cole linearisation, Green-function machinery, and piecewise-constant (PWC) analytic formulas follow \cite{behjoo_harmonic_2025,chertkov_adaptive_2025,chertkov_generative_2025}.

\textbf{Simulations.} Scalar DR scenarios (Scenarios~A and~B): $B=8{,}000$ particles, $n_\mathrm{steps}=2{,}500$ Euler--Maruyama steps on $[0,1]$. Multi-zone scalability experiments ($d$-sweep, $K$-sweep, AR coupling): $B=4{,}000$--$6{,}000$ particles, $n_\mathrm{steps}=2{,}500$. In all cases the MF guidance is set analytically via \eqref{eq:nu-linear}. The $\beta$-schedule is geometrically decreasing: $\beta_j = \beta_0\gamma^{j-1}$ with $\beta_0=12.0$, $\gamma=0.65$, $M=8$ intervals.

\textbf{Demand-response parameters.} \emph{Scalar ($d=1$) baseline.} Target mixture: $\pi=(0.60,0.40)$, $m^{(\mathrm{tar})}=(0.0,1.5)$, $\sigma^{(\mathrm{tar})}=(0.20,0.30)$. Scenario~A: $m^{(\mathrm{in})}=(1.0,6.0)$, $\sigma^{(\mathrm{in})}=(3.0,3.0)$. Scenario~B: $m^{(\mathrm{in})}=(1.5,5.5)$, $\sigma^{(\mathrm{in})}=(0.5,0.7)$. Units: 0 $\widehat{=}$ 20\textdegree C, one unit $\widehat{=}$ 3\textdegree C deviation. Global means: $\bar m^{(\mathrm{in})}_A = 2.8$, $\bar m^{(\mathrm{in})}_B = 2.3$, $\bar m^{(\mathrm{tar})} = 0.6$; MF guidance $\nu_t^{(\MF)} = 2.8 - 2.2t$ (A) and $2.3 - 1.7t$ (B).

\emph{$d$-sweep.} Zone heterogeneity encoded via $z_j = \sin(2\pi j/d)$; target mode means $m^{(\mathrm{tar})}_0 = 0.1\cdot\mathbf{1} + 0.15z$, $m^{(\mathrm{tar})}_1 = 1.5\cdot\mathbf{1} - 0.15z$; initial means displaced by $+1.5$ (occupied) and $+4.0$ (unoccupied); diagonal covariances. $d \in \{1,2,4,8,16,32\}$; $B=4{,}000$.

\emph{$K$-sweep.} $d=4$; $K$ target component means uniformly spaced on $[-1,2]$; initial means $= $ target means $+4$; weights $\propto (K,K{-}1,\ldots,1)$. By construction $\bar m^{(\mathrm{tar})}=0$ for all $K$, so the two IA baselines ($\nu\doteq 0$ and $\nu\doteq\bar m^{(\mathrm{tar})}$) are identical; the comparison reduces to MF vs.\ a single constant-guidance IA.

\emph{AR(1) coupling.} $d=8$, $K=2$, Scenario~B parameters; covariance $\Sigma^{(k)}_{ij} = \sigma_k^2\,\rho^{|i-j|}$, $\rho\in\{0.0, 0.5, 0.8\}$.

Code is available at \url{https://github.com/mchertkov/MeanFieldPID}.

\section*{Acknowledgements}

The author thanks the University of Arizona start-up programme for financial support. This work was initiated during sabbatical visits to the University of Michigan Institute for Computational Discovery and Engineering, the International Centre for Theoretical Physics (ICTP), the Technische Universit\"at Ilmenau (Humboldt Fellowship), Lawrence Livermore National Laboratory (faculty mini-sabbatical program), and KAIST Graduate School of AI. Scientific engagement and encouragement from colleagues at all five institutions are gratefully acknowledged. Large language models (Claude, Anthropic; ChatGPT, OpenAI) assisted with text editing and code refactoring; all mathematical derivations, scientific claims, and code were independently verified by the author.



\newpage 
\appendix  

\leftline{\bf \Large Supplementary Information}

\noindent\textbf{Overview.} This Supplementary Information (SI) provides full mathematical derivations and implementation details supporting the main text. SI\, \ref{SI:theory} 
derives the MF-PID governing equations in both terminal-cost and SOT formulations. SI\, \ref{SI:lqg}
presents the complete LQG closed-form reduction, including the scalar TCL example and explicit performance metrics. SI\,\ref{SI:quad}
develops the theory for the quadratic interaction potential: it first reduces MF-PID to a self-consistent guided H-PID, then proves the central result that the MF guidance is the \emph{exact linear interpolant} between initial and target means (Theorem~\ref{thm:linear-guidance}), and finally assembles the fully explicit Gaussian-mixture score function and marginal density. SI\, \ref{SI:dr}
provides additional experimental diagnostics for the demand-response application. SI\, \ref{SI:ia}
recalls the independent-agent PID foundation \cite{behjoo_harmonic_2025,chertkov_adaptive_2025,chertkov_generative_2025} used throughout.

\section{MF-PID: Governing Equations}
\label{SI:theory}

\subsection{Agent dynamics and mean-field limit}
\label{SI:theory:setup}

Consider $N$ exchangeable agents with dynamics
\begin{equation}
  \dd x_t^{(i)}
  = \bigl(f_t(x_t^{(i)}) + u_t^{(i)}\bigr)\dd t + \dd W_t^{(i)},
  \quad x_0^{(i)} = 0,
  \quad t\in[0,1],
  \label{SI:eq:sde-agent}
\end{equation}
where $x_t^{(i)}\in\Rbb^d$, $f_t$ is a pre-specified base drift, $u_t^{(i)}\in\Rbb^d$ is the control, and $W_t^{(i)}$ are independent standard Brownian motions. The per-agent density $p_t^{(i)}$ satisfies the Kolmogorov--Fokker--Planck (KFP) equation
\begin{equation}
  \partial_t p_t^{(i)}
  + \nabla\cdot\bigl(p_t^{(i)}(f_t + u_t^{(i)})\bigr)
  = \tfrac{1}{2}\Delta p_t^{(i)},
  \quad p_0^{(i)} = \delta(x).
  \label{SI:eq:kfp-per-agent}
\end{equation}
As $N\to\infty$ $p_t^{(N)} \doteq  N^{-1}\sum_i p_t^{(i)}$ converges to the mean-field marginal $p_t$ governed by
\begin{equation}
  \partial_t p_t + \nabla\cdot\bigl(p_t(f_t + u_t)\bigr)
  = \tfrac{1}{2}\Delta p_t,
  \quad p_0 = \delta(x),
  \label{SI:eq:kfp-mf}
\end{equation}
where $u_t = u_t(x, p_t(\cdot))$. A representative agent then obeys the McKean--Vlasov SDE
(main text, eq.~(2)).

\subsection{Formulation A: Terminal-cost MF-PID}
\label{SI:theory:terminal-cost}

The mean-field cost-to-go is
\begin{equation}
  J_t(x, p_t) = \inf_{u_{t\to1}} \mathbb{E}\!\left[
    \int_t^1\!\!\left(\tfrac{1}{2}\|u_{t'}\|^2
    + \int V_{t'}(x_{t'}-y)\,p_{t'}(y)\,\dd y\right)\dd t'
    + \varphi(x_1)
    \;\Big|\;
    \text{eqs.~\eqref{SI:eq:kfp-mf}, \eqref{SI:eq:sde-agent}},\,
    x_t = x
  \right],
  \label{SI:eq:cost-J}
\end{equation}
where $V_t(x-y)$ is an interaction potential and $\varphi$ is a terminal cost.

We work in the classical mean-field control setting where the representative agent solves a control problem with coefficients depending on the population law $p_t$; at equilibrium $p_t$ is generated by the optimal control (self-consistency). Then the mean-field HJB equation is
\begin{equation}
  -\partial_t J_t
  = \underbrace{\int V_t(x-y)\,p_t(y)\,\dd y}_{\Veff_t(x)}
    + f_t(x)\cdot\nabla_x J_t
    + \tfrac{1}{2}\Delta_x J_t
    - \tfrac{1}{2}\|\nabla_x J_t\|^2,
  \quad J_1 = \varphi.
  \label{SI:eq:hjb-mf}
\end{equation}

\paragraph{Hopf--Cole linearization (exact - not an approximation).} Setting $J_t = -\log\psi_t$ in \eqref{SI:eq:hjb-mf} yields the mean-field linear HJB (MF-lin-HJB) equation:
\begin{equation}
  -\partial_t\psi_t + \Veff_t(x)\,\psi_t
  = f_t(x)\cdot\nabla_x\psi_t + \tfrac{1}{2}\Delta_x\psi_t,
  \quad \psi_1 = e^{-\varphi}.
  \label{SI:eq:lin-hjb}
\end{equation}
\begin{remark}[Linear conditional on $p_t$]
\eqref{SI:eq:lin-hjb} is linear in $\psi$
for a given effective potential $\Veff_t$, but the overall MF system
remains nonlinear through the self-consistency $\Veff_t(x)=\int
V_t(x-y)\,p_t(y)\,\dd y$.
\end{remark}

The optimal control and the optimal marginal density are
\begin{align}
  u_t^*(x) &= \nabla_x\log\psi_t(x),
  \label{SI:eq:u-star-terminal}\\
  \partial_t p_t^*
    + \nabla\cdot\bigl(p_t^*(f_t + u_t^*)\bigr)
  &= \tfrac{1}{2}\Delta p_t^*,
  \quad p_0^* = \delta(x).
  \label{SI:eq:kfp-optimal}
\end{align}

\paragraph{Green-function representation.} Because \eqref{SI:eq:lin-hjb} and \eqref{SI:eq:kfp-optimal} are (quasi-) linear in $\psi$ and $p$ respectively, their solutions can be expressed via \emph{effective Green functions} $G_t^{(-;\mfeff)}(x;y)$ and $G_t^{(+;\mfeff)}(x;y)$ satisfying
\begin{align}
  &t\in[1\to0]:\quad
  -\partial_t G_t^{(-;\mfeff)} + \Veff_t(x)\,G_t^{(-;\mfeff)}
  = f_t(x)\cdot\nabla_x G_t^{(-;\mfeff)}
    + \tfrac{1}{2}\Delta_x G_t^{(-;\mfeff)},
  \label{SI:eq:green-minus}\\
  &\hspace{3cm} G_1^{(-;\mfeff)}(x;y) = \delta(x-y), \nonumber\\
  &t\in[0\to1]:\quad
  \partial_t G_t^{(+;\mfeff)} + \Veff_t(x)\,G_t^{(+;\mfeff)}
  = -\nabla_x \cdot \left( f_t(x) G_t^{(+;\mfeff)}\right)
    + \tfrac{1}{2}\Delta_x G_t^{(+;\mfeff)},
  \label{SI:eq:green-plus}\\
  &\hspace{3cm} G_0^{(+;\mfeff)}(x;y) = \delta(x-y). \nonumber
\end{align}
Then
\begin{equation}
  \psi_t(x) = \int e^{-\varphi(y)}\,G_t^{(-;\mfeff)}(x;y)\,\dd y,
  \qquad
  p_t^*(x)
  = G_t^{(+;\mfeff)}(x;0)\,\frac{\psi_t(x)}{\psi_0(0)}.
  \label{SI:eq:psi-p-green}
\end{equation}

The MF coupling enters through $\Veff_t(x) = \int V_t(x-y)\,p_t(y)\,\dd y$, which depends on $p_t$ itself. Eqs.~\eqref{SI:eq:kfp-optimal}--\eqref{SI:eq:green-plus} therefore constitute a \emph{self-consistent} system: in general and unlike the independent-agent case, the Green functions cannot be computed independently of the population density. 

\begin{remark}[Heads up: Quadratic Isotropic Potential] We will see below in Section \ref{SI:quad:linear} that in the case of a quadratic potential $V_t(x)=\beta_t x^2/2$ and zero basic drift $f_t=0$, the MF reduces to guided H-PID with explicit guidance -- an independent-agent case with the linear-in-time interpolant between initial mean and target mean taken as a guidance $\nu_t$ within H-PID with the potential $V_t(x)=\beta_t (x-\nu_t)^2/2$.
\end{remark}

\subsection{Formulation B: Stochastic Optimal Transport (SOT)}
\label{SI:theory:sot}

Replace the terminal cost $\varphi$ with a hard marginal constraint
$p_1 = \ptar$.
The optimal control becomes
\begin{equation}
  u_t^*(x)
  = \nabla_x\log\int\!\!\ptar(y)\,
      \frac{G_t^{(-;\mfeff)}(x;y)}{G_1^{(+;\mfeff)}(y;0)}\,\dd y,
  \label{SI:eq:u-star-SOT}
\end{equation}
and the density evolves under \eqref{SI:eq:kfp-optimal} with $p_0^*=\delta(\cdot)$ and $p_1^*=\ptar$. The governing system is \eqref{SI:eq:green-minus}--\eqref{SI:eq:green-plus} with \eqref{SI:eq:u-star-SOT} substituted into \eqref{SI:eq:kfp-optimal}.

Equivalently, $$p_t^*(x)\propto \psi_t(x) G_t^{(+;\mfeff)}(x;0),$$ where $\psi_t$ solves the backward equation \eqref{SI:eq:lin-hjb}.

\begin{remark}[Comparison with mean-field Schr\"odinger bridges]
The mean-field Schr\"odinger problem studied in \cite{backhoff-veraguas_adapted_2020,hernandez_propagation_2024} corresponds to
$f_t=0$, $V_t=0$ with endpoint constraints.
Our formulation allows an arbitrary base drift $f_t$ and an explicit
running interaction potential $V_t(x-y)$, and does not reduce to the
previously studied mean-field Schr\"odinger systems.
\end{remark}

\section{LQG MF-PID: Closed-Form Reduction}
\label{SI:lqg}

\subsection{Model specification}
\label{SI:lqg:model}

We specialize to:
\begin{itemize}
  \item \textbf{Linear drift:} $f_t(x) = A_t x + B_t$,
    $A_t\in\Rbb^{d\times d}$, $B_t\in\Rbb^d$.
  \item \textbf{Quadratic interaction:} $V_t(x) = \tfrac{1}{2}x^\top Q_t x$,
    $Q_t\succeq 0$.
  \item \textbf{Gaussian target:}
    $\ptar(x) = \calN(x;\,m_1^{(\mathrm{tar})},\Sigma_1^{(\mathrm{tar})})$,
    $\Sigma_1^{(\mathrm{tar})}\succ 0$.
\end{itemize}
Assuming $p_0=\delta(\cdot)$, the controlled SDE \eqref{SI:eq:sde-agent}
with linear optimal control preserves Gaussianity:
$p_t = \calN(x;\,m_t,\Sigma_t)$ for all $t$.

\subsection{Effective potential and quadratic ansatz}
\label{SI:lqg:eff-potential}

Substituting the Gaussian $p_t$ into $\Veff_t$:
\begin{equation}
  \Veff_t(x)
  = \tfrac{1}{2}(x-m_t)^\top Q_t(x-m_t)
    + \tfrac{1}{2}\tr(Q_t\Sigma_t).
  \label{SI:eq:Veff-lqg}
\end{equation}
The trace term is $x$-independent and does not affect $\nabla_x J_t$ or
the control.
We seek a quadratic cost-to-go
$J_t(x) = \tfrac{1}{2}x^\top S_t x + s_t^\top x + r_t$,
giving the affine optimal control
\begin{equation}
  u_t^*(x) = -S_t x - s_t.
  \label{SI:eq:u-star-lqg}
\end{equation}

\subsection{The closed ODE system}
\label{SI:lqg:odes}

Substituting \eqref{SI:eq:Veff-lqg} and the quadratic ansatz into the
MF-HJB \eqref{SI:eq:hjb-mf} and matching polynomial terms in $x$
yields the following closed system.

\begin{proposition}[LQG MF-PID reduction]
\label{prop:lqg}
Under the LQG model specification of \S\ref{SI:lqg:model}, the
optimal mean-field cost-to-go is quadratic and the four coupled ODEs
\begin{tcolorbox}[colback=gray!8, colframe=gray!40,
  title={\color{black} LQG MF-PID Equations}]
\begin{align}
  -\dot S_t &= Q_t + S_t A_t + A_t^\top S_t - S_t^2,
  \tag{S2.1}\label{SI:eq:riccati-S}\\
  -\dot s_t &= -Q_t m_t + A_t^\top s_t - S_t s_t + S_t B_t,
  \tag{S2.2}\label{SI:eq:riccati-s}\\
  \dot m_t &= (A_t - S_t)m_t + B_t - s_t,
  \tag{S2.3}\label{SI:eq:mean-ode}\\
  \dot\Sigma_t &= (A_t-S_t)\Sigma_t + \Sigma_t(A_t-S_t)^\top + I,
  \tag{S2.4}\label{SI:eq:cov-ode}
\end{align}
\end{tcolorbox}
with boundary conditions
$m_0=0$, $\Sigma_0=0$,
$m_1=m_1^{(\mathrm{tar})}$, $\Sigma_1=\Sigma_1^{(\mathrm{tar})}$,
characterise the unique optimal control via
$u_t^*(x) = -S_t x - s_t$.
\end{proposition}

\paragraph{Structural decoupling.}
Equations \eqref{SI:eq:riccati-S}--\eqref{SI:eq:cov-ode} separate
into two sequential sub-problems:
\begin{enumerate}[1.]
  \item \textbf{Variance block} $(S_t, \Sigma_t)$:
    Equations \eqref{SI:eq:riccati-S} and \eqref{SI:eq:cov-ode} depend
    only on $Q_t$ and $A_t$, and are \emph{identical} in the MF and IA
    cases.
    They are solved first as a two-point boundary-value problem in
    $\Sigma$ (shooting on $S_1$).
  \item \textbf{Mean block} $(m_t, s_t)$:
    Given $S_t$, equations \eqref{SI:eq:riccati-s}--\eqref{SI:eq:mean-ode}
    are linear in $(m_t, s_t)$ with boundary conditions $m_0=0$,
    $m_1=m_1^{(\mathrm{tar})}$.
    The MF and IA cases differ here: the MF source term is $-Q_t m_t$
    (coupling to the \emph{current} population mean), whereas the IA
    source is $-Q_t\bar m$ (a fixed exogenous centre).
\end{enumerate}
This decoupling is the key structural result:
\emph{MF coordination operates entirely through the mean block}.

\subsection{Scalar TCL example: complete closed-form solution}
\label{SI:lqg:tcl}

Set $d=1$, $A_t=-\kappa$ ($\kappa>0$), $B_t=0$, $Q_t=q$ ($q\geq0$),
$\ptar = \calN(m^{(\mathrm{tar})},(\sigma^{(\mathrm{tar})})^2)$.

\subsubsection{Step 1: Riccati equation}
With $\Delta\doteq\sqrt{\kappa^2+q}$, the scalar Riccati equation
$\dot S_t = S_t^2 + 2\kappa S_t - q$ is solved by
\begin{equation}
  S_t = -\kappa + \Delta\,
  \frac{1 + \rho\,e^{-2\Delta(1-t)}}{1 - \rho\,e^{-2\Delta(1-t)}},
  \label{SI:eq:St-scalar}
\end{equation}
where $\rho = (S_1+\kappa-\Delta)/(S_1+\kappa+\Delta)$.

\subsubsection{Step 2: Variance matching (bridge constraint)}
The covariance ODE $\dot\Sigma_t = -2(\kappa+S_t)\Sigma_t + 1$,
$\Sigma_0=0$ has the explicit solution
\begin{equation}
  \Sigma_1 = \frac{1-\rho}{1-\rho r_0}\cdot\frac{1-r_0}{2\Delta},
  \quad r_0 = e^{-2\Delta}.
  \label{SI:eq:Sigma1}
\end{equation}
Imposing $\Sigma_1 = (\sigma^{(\mathrm{tar})})^2$ yields the
closed-form shooting parameter:
\begin{equation}
  \rho = \frac{A-1}{A r_0 - 1},
  \quad
  A = \frac{2\Delta(\sigma^{(\mathrm{tar})})^2}{1-r_0},
  \label{SI:eq:rho-scalar}
\end{equation}
and then $S_1 = -\kappa + \Delta(1+\rho)/(1-\rho)$.

\subsubsection{Step 3: Mean dynamics}
With $S_t$ determined, eliminating $s_t$ from
\eqref{SI:eq:riccati-s}--\eqref{SI:eq:mean-ode} yields
$\ddot m_t - \kappa^2 m_t = 0$,
whose solution satisfying $m_0=0$, $m_1=m^{(\mathrm{tar})}$ is
\begin{equation}
  m_t^{(\MF)}
  = m^{(\mathrm{tar})}\,\frac{\sinh(\kappa t)}{\sinh(\kappa)}.
  \label{SI:eq:mt-mf}
\end{equation}
The linear coefficient is then
\begin{equation}
  s_t = -\dot m_t - (\kappa + S_t)\,m_t
  = -\frac{\kappa\,m^{(\mathrm{tar})}\cosh(\kappa t)}{\sinh(\kappa)}
    - (\kappa + S_t)
      \frac{m^{(\mathrm{tar})}\sinh(\kappa t)}{\sinh(\kappa)}.
  \label{SI:eq:st-mf}
\end{equation}

\subsubsection{Step 4: IA baseline}
For the IA family with exogenous centre $\bar m$,
\eqref{SI:eq:riccati-s} becomes
$\dot s_t^{(\IA)} = q\bar m + (\kappa+S_t)s_t^{(\IA)}$,
with solution
\begin{equation}
  s_t^{(\IA)} = g(t)\bigl(s_0^{(\IA)} + q\bar m\,J(t)\bigr),
  \label{SI:eq:st-ia}
\end{equation}
where $g(t) = e^{\Delta t}(1-\rho e^{-2\Delta})/(1-\rho e^{2\Delta(t-1)})$
and $J(t) = (1-e^{-2\Delta t}-\rho e^{-2\Delta}(e^{2\Delta t}-1))/(2\Delta(1-\rho e^{-2\Delta}))$.
The initial condition $s_0^{(\IA)}$ is fixed by the bridge constraint
$m_1^{(\IA)} = m^{(\mathrm{tar})}$:
\begin{equation}
  s_0^{(\IA)} = -\frac{m^{(\mathrm{tar})} + q\bar m\,\tilde B(\rho)}{\tilde A(\rho)},
  \quad
  \tilde A(\rho) = g(1)^{-1}\!\int_0^1\!\!g(u)^2\,\dd u,
  \quad
  \tilde B(\rho) = g(1)^{-1}\!\int_0^1\!\!g(u)^2 J(u)\,\dd u.
  \label{SI:eq:s0-ia}
\end{equation}

\subsubsection{Performance metrics}
Three metrics differentiate MF from IA:
\begin{align}
  \mu_u(t) &= -(S_t m_t + s_t), \quad
  \sigma_u^2(t) = S_t^2\sigma_t^2
    && \text{(control mean and variance)},\\
  \calP(t) &= S_t^2\sigma_t^2 + (S_t m_t + s_t)^2
    && \text{(instantaneous power)},\\
  \calE(t) &= \int_0^t \calP(u)\,\dd u
    && \text{(cumulative energy)}.
\end{align}
Since $S_t$ and $\sigma_t^2$ are identical across schemes, the energy
difference arises exclusively from the linear coefficient $s_t$.
The MF controller achieves strictly lower $\calE(1)$ than any
exogenously centred IA strategy.

\section{MF-PID with Quadratic Interaction Potential}
\label{SI:quad}

\subsection{Reduction to self-consistent guided H-PID}
\label{SI:quad:reduction}

We now consider MF-PID in the case of an isotropic quadratic interaction
potential
\begin{equation}
  V_t(x) = \frac{\beta_t}{2}\|x\|^2, \quad \beta_t > 0,
  \label{SI:eq:V-quadratic}
\end{equation}
with \emph{arbitrary} base drift $f_t$ and \emph{arbitrary} initial and target distributions. Computing the effective potential \eqref{SI:eq:cost-J} gives
\begin{equation}
  \Veff_t(x) = \int V_t(x-y)\,p_t(y)\,\dd y= \frac{\beta_t}{2}\|x - m_t\|^2 + \frac{\beta_t}{2}\tr\Sigma_t,
  \label{SI:eq:Veff-quad}
\end{equation}
where $m_t = \mathbb{E}_{p_t}[x]$ is the population mean and $\Sigma_t = \mathrm{Cov}_{p_t}[x]$. The trace term is $x$-independent and does not affect the control. Defining the guidance centre $\nu_t \doteq m_t$, the effective potential reduces to $\Veff_t(x) = \tfrac{\beta_t}{2}\|x-\nu_t\|^2 + \text{const}$, which is the quadratic guidance potential of the H-PID framework \cite{behjoo_harmonic_2025,chertkov_generative_2025}.

\begin{proposition}[MF-PID as self-consistent guided H-PID]
\label{prop:mf-hpid}
For the quadratic potential \eqref{SI:eq:V-quadratic}, MF-PID is equivalent to a guided H-PID whose guidance trajectory $\nu_t = \nu_t^{(\MF)}$ is determined endogenously by the self-consistency condition
\begin{equation}
  \nu_t^{(\MF)}= \mathbb{E}_{x\sim p_t^{*}(\,\cdot\,|\,\nu^{(\MF)})}[x], \quad t\in[0,1].
  \label{SI:eq:self-consistency}
\end{equation}
In the SOT formulation, the corresponding optimal control is
\begin{equation}
  u_t^*(x)= \nabla_x\log\int\!\!\ptar(y)\,\frac{G_t^{(-)}(x;y\,|\,\nu^{(\MF)})}{G_1^{(+)}(y;0\,|\,\nu^{(\MF)})}\,\dd y,
  \label{SI:eq:u-star-quad}
\end{equation}
where $G_t^{(\pm)}(\,\cdot\,|\,\nu^{(\MF)})$ are the Green functions of the H-PID system evaluated under the guidance $\nu^{(\MF)}$.
\end{proposition}

\subsection{Linearity of the MF guidance}
\label{SI:quad:linear}

The following theorem is the central analytical result for the
zero-drift case.

\begin{theorem}[Linear MF guidance]
\label{thm:linear-guidance}
Let $f_t \doteq  0$, let $V_t$ be the quadratic potential \eqref{SI:eq:V-quadratic} with an arbitrary schedule $\beta_t > 0$, and let $\pin$, $\ptar$ be any probability measures with finite first moments $\bar m^{(\mathrm{in})} = \mathbb{E}_{\pin}[x]$ and $\bar m^{(\mathrm{tar})} = \mathbb{E}_{\ptar}[x]$. Also assume that the controlled process is initialized with $x_0\sim p^{(\mathrm{in})}$ (i.e. $p_0=p^{(\mathrm{in})}$), and that the MF fixed point exists and yields finite first moments  $m_t = \mathbb{E}_{p_t^*}[x]$ for all $t$. Then
\begin{equation}
  \nu_t^{(\MF)} = m_t
  = (1-t)\,\bar m^{(\mathrm{in})} + t\,\bar m^{(\mathrm{tar})}
  \quad \text{for all } t\in[0,1].
  \label{SI:eq:nu-linear}
\end{equation}
That is, the self-consistent MF guidance is the \emph{exact} linear interpolant between initial and target global means, independently of $\beta_t$ -- for any measurable $\beta_t>0$ such that the MF bridge is well-posed and $\mathbb E\|x_t\|<\infty$ for all $t$ -- and of the shape of $\pin$ and $\ptar$.
\end{theorem}

\begin{proof}
We show that $\ddot m_t = 0$.
\medskip

\noindent\textbf{Step 1 (It\^o + mean acceleration).}
With $f_t\doteq 0$, the McKean--Vlasov SDE is
$\dd x_t = u_t^*(x_t)\,\dd t + \dd W_t$.
Differentiating $m_t = \mathbb{E}[x_t]$ gives
$\dot m_t = \mathbb{E}[u_t^*(x_t)]$.
Applying It\^o's formula to $u_t^*(x_t)$ and taking expectations:
\begin{equation}
  \ddot m_t
  = \mathbb{E}\!\left[
      \partial_t u_t^*
      + (u_t^*\cdot\nabla_x)u_t^*
      + \tfrac{1}{2}\Delta u_t^*
    \right].
  \label{SI:eq:proof-ddot-m}
\end{equation}

\noindent\textbf{Step 2 (Differentiated HJB in space).} With $J_t = -\log\psi_t$ and $u_t^* = \nabla_x\log\psi_t$, the HJB equation \eqref{SI:eq:hjb-mf} with $f_t=0$ reads
\begin{equation*}
  \partial_t J_t= -\Veff_t + \tfrac{1}{2}|u_t^*|^2 + \tfrac{1}{2}\Delta\log\psi_t.
\end{equation*}
Taking the gradient in $x$ and using the symmetry of $\nabla_xu_t^* = \mathrm{Hess}(\log\psi_t)$:
\begin{equation}
  \partial_t u_t^*= \nabla_x\Veff_t- \underbrace{(\nabla_xu_t^*)\,u_t^*}_{=(u_t^*\cdot\nabla_x)u_t^*}- \tfrac{1}{2}\Delta u_t^*.
  \label{SI:eq:proof-dtu}
\end{equation}

\noindent\textbf{Step 3 (Cancellation).} Substituting \eqref{SI:eq:proof-dtu} into \eqref{SI:eq:proof-ddot-m}, the $(u_t^*\cdot\nabla_x)u_t^*$ and $\tfrac{1}{2}\Delta u_t^*$ terms cancel exactly, leaving
\begin{equation*}
  \ddot m_t = \mathbb{E}\!\left[\nabla_x\Veff_t(x_t)\right].
\end{equation*}
This identity holds for any interaction potential; the complex score structure of $u_t^*$ does not appear.

\noindent\textbf{Step 4 (Quadratic potential + self-consistency).} For the quadratic potential \eqref{SI:eq:Veff-quad}:
\begin{equation*}
  \nabla_x\Veff_t(x) = \beta_t(x - \nu_t^{(\MF)}).
\end{equation*}
Therefore
\begin{equation*}
  \ddot m_t
  = \beta_t\,\mathbb{E}[x_t - \nu_t^{(\MF)}]
  = \beta_t\,(m_t - \nu_t^{(\MF)}).
\end{equation*}
At the MF fixed point, the self-consistency condition
\eqref{SI:eq:self-consistency} gives $\nu_t^{(\MF)} = m_t$, so
\begin{equation*}
  \ddot m_t = \beta_t\,(m_t - m_t) = 0.
\end{equation*}

\noindent\textbf{Conclusion.}
With $\ddot m_t = 0$ and boundary conditions
$m_0 = \bar m^{(\mathrm{in})}$, $m_1 = \bar m^{(\mathrm{tar})}$,
we obtain $m_t = (1-t)\,\bar m^{(\mathrm{in})} + t\,\bar m^{(\mathrm{tar})}$,
and since $\nu_t^{(\MF)} = m_t$, the claim \eqref{SI:eq:nu-linear} follows.
\end{proof}

\begin{remark}[Independence of $\beta_t$, initial law, and target] The proof uses only three ingredients: the It\^o--HJB cancellation (Step 3), which is a structural identity holding for any smooth $u_t^*$; the linearity of $\nabla_x\Veff_t$ in $x$ (Step 4), which is a consequence of the quadratic form of $V_t$; and the self-consistency relation $\nu_t^{(\MF)} = m_t$. The $\beta_t$ schedule, the shape of $\pin$ and $\ptar$ enter only through the score function $u_t^*$, which has already canceled out. Theorem~\ref{thm:linear-guidance} therefore holds for arbitrary $\pin$ and $\ptar$  and for any  $\beta_t$ resulting in a well-defined densities.
\end{remark}

\begin{corollary}[Linearity for delta initial condition]
\label{cor:delta-IC}
When $\pin = \delta(\cdot)$ (i.e.\ $\bar m^{(\mathrm{in})}=0$),
Theorem~\ref{thm:linear-guidance} gives
$\nu_t^{(\MF)} = t\,\bar m^{(\mathrm{tar})}$.
\end{corollary}

\begin{remark}[Relation to the Brownian and Schr\"{o}dinger bridge]
\label{rem:schrodinger}
Setting $V_t \doteq  0$ (i.e.\ $\beta_t = 0$) in the H-PID framework removes the interaction potential entirely, recovering the standard \emph{Schr\"{o}dinger bridge} \cite{leonard_survey_2013}: the problem of finding the most likely path of a Brownian motion that transports $\pin$ to $\ptar$. When both marginals are delta distributions, this further specialises to the classical \emph{Brownian bridge} (a Brownian motion pinned at both endpoints). In both cases the mean trajectory $m_t = \mathbb{E}[x_t]$ is the linear interpolant $(1{-}t)\bar m^{(\mathrm{in})} + t\bar m^{(\mathrm{tar})}$: for the Brownian bridge this is immediate from the explicit formula $\mathbb{E}[x_t] = (1{-}t)x_0 + t x_1$; for the general Schr\"{o}dinger bridge it follows from the affine structure of the Doob $h$-transform (see \cite{leonard_survey_2013}, Remark~1.8).

Theorem~\ref{thm:linear-guidance} is a strictly stronger statement. It establishes the \emph{same} linear interpolation for \emph{any} $\beta_t > 0$, however large or time-varying, and for arbitrary non-Gaussian, multi-modal marginals $\pin$ and $\ptar$. In the limit $\beta_t \to 0$ our result trivially reproduces the Schr\"{o}dinger bridge case (the mean acceleration equation $\ddot m_t = \beta_t(m_t - \nu_t^{(\MF)}) \to 0$ degenerately), but the proof mechanism is completely different: it does not rely on the $h$-transform or any special structure of the bridge kernel. Instead, it rests on the Itô--HJB cancellation (Step~3 of the proof), which is a structural identity for the \emph{interacting} system, and holds precisely because the quadratic potential makes $\nabla_x V^{(\mathrm{eff})}_t$ linear in $x$. The non-interacting Schr\"{o}dinger bridge therefore corresponds to the special limit of our theorem in which the interaction is switched off, not the other way around.

A related observation holds in the stochastic interpolant framework \cite{albergo_building_2023}, where the base interpolant $I_t=(1-t)x_0+t x_1+\sigma z$ has mean $(1{-}t)\mathbb{E}[x_0]+t\mathbb{E}[x_1]$ by construction; in MF-PID with $V_t>0$, the same linearity is not imposed but \emph{derived} from the Itô--HJB cancellation, and holds for any $\beta_t$ and any initial and final densities.
\end{remark}

The following corollary recovers the LQG result of \S\ref{SI:lqg:tcl} as a special case and identifies the precise role of the base drift.

\begin{corollary}[OU base drift]
\label{cor:OU}
For $f_t(x) = A x$ with $A = -\kappa I$ and $p_0=\delta(\cdot)$, the same It\^o--HJB cancellation in Steps~1--3 applies.

More generally, for a base drift $f_t$ (under sufficient regularity), Steps~1--3 yield the identity
\[
\ddot m_t
= \frac{d}{dt}\mathbb E[f_t(x_t)]
  + \mathbb E[\nabla \Veff_t(x_t)]
  - \mathbb E[(\nabla f_t(x_t))^\top u_t^*(x_t)].
\]
For linear $f_t(x)=A x$ (constant Jacobian $\nabla f_t = A$) and quadratic interaction, $\mathbb E[\nabla \Veff_t]=\beta_t(m_t-\nu_t^{(\MF)})$ vanishes at the MF fixed point. Using $\dot m_t = A m_t + \mathbb E[u_t^*(x_t)]$ (hence $\mathbb E[u_t^*(x_t)] = \dot m_t - A m_t$) gives
\[
\ddot m_t = (A-A^\top)\dot m_t + A^\top A\,m_t.
\]
Specializing to $A=-\kappa I$ yields $\ddot m_t=\kappa^2 m_t$ and hence, with $m_0=0$ and $m_1=\bar m^{(\mathrm{tar})}$,
\begin{equation}
  \nu_t^{(\MF)} = m_t^{(\MF)}= \bar m^{(\mathrm{tar})}\,\frac{\sinh(\kappa t)}{\sinh(\kappa)}.
  \label{SI:eq:nu-OU}
\end{equation}
This departs from the linear interpolant by $O(\kappa^2)$; Eq.~\eqref{SI:eq:mt-mf} is recovered.
\end{corollary}

\begin{remark}[Practical implication] Theorem~\ref{thm:linear-guidance} provides the MF guidance \emph{in closed form}, without iteration: for $f_t\doteq 0$, one sets $\nu_t^{(\MF)} = (1-t)\,\bar m^{(\mathrm{in})} + t\,\bar m^{(\mathrm{tar})}$ and proceeds directly to the score function computation of
\S\ref{SI:quad:score}.
The self-consistency fixed-point iteration is therefore unnecessary in this case and is included only for reference and for the $f_t \not\doteq 0$ setting.
\end{remark}

\subsection{Green-function solution: PWC protocol}
\label{SI:quad:pwc}

Following \cite{chertkov_adaptive_2025,chertkov_generative_2025}, we discretise $[0,1]$ by
$0=t_0<t_1<\cdots<t_M=1$ and represent the protocol by piecewise-constant
(PWC) values $(\beta_i,\nu_i)$ on each interval $[t_i,t_{i+1})$.
By Theorem~\ref{thm:linear-guidance}, for $f_t\doteq  0$:
\begin{equation}
  \nu_i = (1-t_i)\,\bar m^{(\mathrm{in})} + t_i\,\bar m^{(\mathrm{tar})},
  \quad i=0,\ldots,M.
  \label{SI:eq:nu-pwc-explicit}
\end{equation}
The guided Green functions take the Gaussian form
\begin{align}
  G_t^{(-)}(x|y)
  &\propto\exp\!\Bigl(
    -\tfrac{a_t^{(-)}}{2}\|x-\nu_t\|^2
    + b_t^{(-)}(x-\nu_t)^\top(y-\nu_t) \nonumber\\
  &\hspace{2cm}
    - \tfrac{c_t^{(-)}}{2}\|y-\nu_t\|^2
    + (r_t^{(-)})^\top(x-\nu_t)
    + (s_t^{(-)})^\top(y-\nu_t)\Bigr),
  \label{SI:eq:G-minus}\\
  G_t^{(+)}(y|0)
  &\propto\exp\!\Bigl(
    -\tfrac{a_t^{(+)}}{2}\|y-\nu_t\|^2
    + (s_t^{(+)})^\top(y-\nu_t)\Bigr).
  \label{SI:eq:G-plus}
\end{align}
The scalar coefficients satisfy the Riccati equations
\begin{equation}
  \mp\dot a_t^{(\pm)} + \beta_t = (a_t^{(\pm)})^2,
  \qquad
  \dot b_t^{(-)} = a_t^{(-)}b_t^{(-)},
  \qquad
  \dot c_t^{(-)} = (b_t^{(-)})^2.
  \label{SI:eq:riccati-abc}
\end{equation}
Within each PWC interval these admit closed hyperbolic forms
(\S\ref{SI:quad:closedforms}).

\subsection{Score function in closed form}
\label{SI:quad:score}

\begin{proposition}[Gaussian-mixture score]
\label{prop:gm-score}
For a Gaussian-mixture target
$$\ptar = \sum_{k=1}^K \pi_k\calN(x;\,m_k,\Sigma_k),$$
and guidance $\nu_t$ (given by \eqref{SI:eq:nu-pwc-explicit}
for $f_t\doteq 0$),
substituting \eqref{SI:eq:G-minus}--\eqref{SI:eq:G-plus} into the SOT
formula \eqref{SI:eq:u-star-SOT} yields the score function
\begin{equation}
  u_t^*(x) = b_t^{(-)}\bigl(\hat y(t;x) - \Upsilon_t(x)\bigr),
  \label{SI:eq:score-gm}
\end{equation}
where
\begin{align}
  \Upsilon_t(x)
  &= \nu_t + \frac{a_t^{(-)}(x-\nu_t) - r_t^{(-)}}{b_t^{(-)}},
  \label{SI:eq:Upsilon}\\
  \hat y(t;x)
  &= \sum_{k=1}^K \bar\pi_k(t;x)\,\bar m_k(t;x),
  \label{SI:eq:yhat}\\
  \bar m_k(t;x)
  &= \bigl(\Sigma_k^{-1} + K_t I\bigr)^{-1}
     \bigl(\Sigma_k^{-1}m_k + K_t\,\mu_t(x)\bigr),
  \label{SI:eq:mk-bar}\\
  \bar\pi_k(t;x)
  &= \frac{\pi_k\,\calN(\mu_t(x);\,m_k,\,\Sigma_k + K_t^{-1}I)}
          {\sum_\ell\pi_\ell\,\calN(\mu_t(x);\,m_\ell,\,\Sigma_\ell + K_t^{-1}I)},
  \label{SI:eq:pik-bar}
\end{align}
with probe distribution parameters
\begin{equation}
  K_t = c_t^{(-)} - a_1^{(+)},
  \qquad
  \mu_t(x)
  = \frac{b_t^{(-)}}{K_t}(x-\nu_t)
    + \frac{s_t^{(-)} - s_1^{(+)}}{K_t} + \nu_t.
  \label{SI:eq:Kt-mut}
\end{equation}
\end{proposition}

\subsection{Closed-form PWC evolution within each interval}
\label{SI:quad:closedforms}

Fix interval $i$ with $\beta_t=\beta_i$, $\nu_t=\nu_i$, and set
$\omega_i=\sqrt{\beta_i}$, $\tau=t-t_i$.

\paragraph{Scalar coefficients.}
\textit{Forward branch} ($a^{(+)}$, propagated forward from
$t=0^+$ where $a^{(+)}\sim 1/t$):
\begin{equation}
  a^{(+)}(t) = \omega_i\coth(\omega_i\tau + \varphi_i^{(+)}),
  \label{SI:eq:a-plus-pwc}
\end{equation}
with $\varphi_1^{(+)}=0$ and subsequent phases set by continuity.

\textit{Backward branch} ($a^{(-)}$, propagated backward from
$t=1^-$ where $a^{(-)}\sim 1/(1-t)$).
On the terminal interval:
\begin{equation}
  a^{(-)}(t) = c^{(-)}(t) = \omega_{M-1}\coth(\omega_{M-1}(1-t)),
  \quad
  b^{(-)}(t) = \omega_{M-1}\,\mathrm{csch}(\omega_{M-1}(1-t)).
  \label{SI:eq:abc-terminal}
\end{equation}
On earlier interval $i$, with right-endpoint anchors
$(a^{(-)}_{i+1},b^{(-)}_{i+1},c^{(-)}_{i+1})$ and $\tau=t_{i+1}-t$:
\begin{align}
  a^{(-)}(t)
  &= \omega_i\,
     \frac{a^{(-)}_{i+1} + \omega_i\tanh(\omega_i\tau)}
          {\omega_i + a^{(-)}_{i+1}\tanh(\omega_i\tau)},
  \label{SI:eq:a-minus-pwc}\\
  b^{(-)}(t)
  &= b^{(-)}_{i+1}
     \sqrt{\frac{\beta_i-(a^{(-)}(t))^2}
                {\beta_i-(a^{(-)}_{i+1})^2}},
  \label{SI:eq:b-minus-pwc}\\
  c^{(-)}(t)
  &= c^{(-)}_{i+1}
     + \frac{(b^{(-)}_{i+1})^2}{\beta_i-(a^{(-)}_{i+1})^2}
       \bigl(a^{(-)}_{i+1}-a^{(-)}(t)\bigr).
  \label{SI:eq:c-minus-pwc}
\end{align}

\paragraph{Vector (linear) coefficients.}
Introduce the re-centred linear coefficients
\begin{equation}
  \theta_t^{(+)} = s_t^{(+)} + a_t^{(+)}\nu_t,
  \quad
  \theta_{x,t}^{(-)} = r_t^{(-)} + (a_t^{(-)}-b_t^{(-)})\nu_t,
  \quad
  \theta_{y,t}^{(-)} = s_t^{(-)} + (c_t^{(-)}-b_t^{(-)})\nu_t.
  \label{SI:eq:theta-def}
\end{equation}
The ODEs for these quantities are linear with hyperbolic driving terms
from \eqref{SI:eq:a-plus-pwc}--\eqref{SI:eq:c-minus-pwc}.
Their closed-form solutions on each interval are:
\begin{align}
  \theta^{(+)}(t)
  &= \frac{\sinh\varphi_i^{(+)}}
          {\sinh(\omega_i\tau+\varphi_i^{(+)})}\,
     \theta^{(+)}(t_i)
     + \frac{\beta_i}{\omega_i}\,
       \frac{\cosh(\omega_i\tau+\varphi_i^{(+)})-\cosh\varphi_i^{(+)}}
            {\sinh(\omega_i\tau+\varphi_i^{(+)})}\,\nu_i,
  \label{SI:eq:theta-plus-pwc}\\
  \theta_x^{(-)}(t)
  &= \frac{b^{(-)}(t)}{b^{(-)}_{i+1}}\,\theta_x^{(-)}(t_{i+1}^-)
     + \Bigl(a^{(-)}(t)
       - \frac{b^{(-)}(t)}{b^{(-)}_{i+1}}a^{(-)}_{i+1}\Bigr)\nu_i,
  \label{SI:eq:theta-x-pwc}\\
  \theta_y^{(-)}(t)
  &= \theta_y^{(-)}(t_{i+1}^-)
     + \frac{c^{(-)}_{i+1}-c^{(-)}(t)}{b^{(-)}_{i+1}}\,
       \theta_x^{(-)}(t_{i+1}^-)
     + \Bigl((b^{(-)}_{i+1}-b^{(-)}(t))
       - a^{(-)}_{i+1}
         \frac{c^{(-)}_{i+1}-c^{(-)}(t)}{b^{(-)}_{i+1}}\Bigr)\nu_i.
  \label{SI:eq:theta-y-pwc}
\end{align}

\subsection{Time-marginal density in closed form}
\label{SI:quad:density}

\begin{proposition}[Optimal marginal density]
\label{prop:marginal}
Define the derived time-continuous quantities
\begin{equation}
  K_t = c_t^{(-)} - a_1^{(+)},
  \quad
  \alpha_t = b_t^{(-)}/K_t,
  \quad
  \bar d_t = (\theta_{y,t}^{(-)} - \theta_1^{(+)})/K_t,
  \quad
  S_k(t) = \Sigma_k + K_t^{-1}I,
  \label{SI:eq:derived-Kt}
\end{equation}
and
\begin{align}
  M_k(t)
  &= \bigl(a_t^{(+)}+a_t^{(-)}-({b_t^{(-)}}^2/K_t)\bigr)I
     + \alpha_t^2 S_k(t)^{-1},
  \label{SI:eq:Mk}\\
  h_k(t)
  &= \theta_t^{(+)} + \theta_{x,t}^{(-)}
     + b_t^{(-)}\bar d_t
     + \alpha_t S_k(t)^{-1}(m_k-\bar d_t).
  \label{SI:eq:hk}
\end{align}
Then the optimal marginal density is the Gaussian mixture
\begin{equation}
  p_t^{*}(x) \propto
  \sum_{k=1}^K \pi_k |S_k(t)|^{-1/2}
  \exp\!\Bigl(
    -\tfrac{1}{2}x^\top M_k(t)x + h_k(t)^\top x
    - \tfrac{1}{2}(m_k-\bar d_t)^\top S_k(t)^{-1}(m_k-\bar d_t)
  \Bigr).
  \label{SI:eq:p-star-gm}
\end{equation}
\end{proposition}

\subsection{Non-zero initial condition and mixture-to-mixture transport}
\label{SI:quad:nonzero-IC}

\subsubsection{Deterministic start $x_0=z$}
A coordinate shift $\tilde x = x-z$ reduces the problem to a zero-start
problem with shifted guidance $\tilde\nu_t = \nu_t - z$ and shifted
target means $\tilde m_k = m_k - z$.
The scalar Riccati coefficients
$(a_t^{(\pm)}, b_t^{(-)}, c_t^{(-)})$ are unchanged; the linear
coefficients acquire $z$-linear corrections via three scalar shift
propagators $\lambda^{(+)}(t)$, $\lambda_x^{(-)}(t)$,
$\lambda_y^{(-)}(t)$ satisfying the same ODEs as
$(\theta^{(+)},\theta_x^{(-)},\theta_y^{(-)})$
but with the source $\beta_t\nu_t$ replaced by $\beta_t$:
\begin{equation}
  \tilde\theta_t^{(+)} = \theta_t^{(+)} - \lambda^{(+)}(t)\,z,
  \quad
  \tilde\theta_{x,t}^{(-)} = \theta_{x,t}^{(-)} + \lambda_x^{(-)}(t)\,z,
  \quad
  \tilde\theta_{y,t}^{(-)} = \theta_{y,t}^{(-)} + \lambda_y^{(-)}(t)\,z.
  \label{SI:eq:theta-shifted}
\end{equation}
The shift propagators share the same PWC closed forms as
\eqref{SI:eq:theta-plus-pwc}--\eqref{SI:eq:theta-y-pwc}
with $\nu_i$ replaced by~$1$.

\subsubsection{Stochastic initial condition $x_0\sim\pin$}
For a mixture initial condition
$\pin(x)=\sum_{j=1}^J\pi_j^{(\mathrm{in})}\calN(x;\,m_j^{(\mathrm{in})},\Sigma_j^{(\mathrm{in})})$,
sample $z^{(i)}\sim\pin$ per trajectory and apply the shift
transformation above.
By Theorem~\ref{thm:linear-guidance} with $f_t\doteq 0$, the global
mean evolves as
\begin{equation}
  m_t = (1-t)\,\bar m^{(\mathrm{in})} + t\,\bar m^{(\mathrm{tar})},
  \label{SI:eq:mt-gm-explicit}
\end{equation}
so the guidance is fully explicit.
The shifted scalar coefficients and propagators are computed once;
per-trajectory cost is $\mathcal{O}(d)$ additions.

The optimal marginal density is the Gaussian mixture with $J\times K$
components
\begin{equation}
  p_t^{*}(x) \propto
  \sum_{j=1}^J\sum_{k=1}^K \tilde\pi_{j,k}(t)\,
  \calN(x;\,\tilde m_{j,k}(t),\,S_k(t)),
  \label{SI:eq:p-star-mixture-to-mixture}
\end{equation}
where the weights $\tilde\pi_{j,k}$ and means $\tilde m_{j,k}$ are
given by the following Gaussian completing-the-square formulas.

\paragraph{Component weights and means.}
Define the effective precision
$\Lambda_j(t) = (\Sigma_j^{(\mathrm{in})})^{-1} + (a_t^{(+)}-a_1^{(+)})I$
and the linear vector
$\ell_j(t)
= (\Sigma_j^{(\mathrm{in})})^{-1}m_j^{(\mathrm{in})}
  + \theta_t^{(+)} - a_1^{(+)}m_k - \theta_1^{(+)}$.
Then:
\begin{align}
  \tilde m_{j,k}(t)
  &= \bigl(M_k(t)+\Lambda_j(t)\bigr)^{-1}
    \Bigl(h_k(t) + \Lambda_j(t)^{-1}
    \bigl((\Sigma_j^{(\mathrm{in})})^{-1}m_j^{(\mathrm{in})}
      + \theta_t^{(+)} - a_1^{(+)}m_k - \theta_1^{(+)}\bigr)\Bigr),
  \label{SI:eq:m-tilde}\\
  \tilde\pi_{j,k}(t)
  &\propto
    \pi_j^{(\mathrm{in})}\pi_k\,
    |\Lambda_j(t)|^{-1/2}|\Sigma_j^{(\mathrm{in})}|^{-1/2}
    |S_k(t)|^{-1/2}|\tilde M_{j,k}(t)|^{-1/2}
    \exp\!\Bigl(
      \tfrac{1}{2}\tilde h_{j,k}^\top\tilde M_{j,k}^{-1}\tilde h_{j,k}
      - \tilde q_{j,k}(t)
    \Bigr),
  \label{SI:eq:pi-tilde}
\end{align}
where
$\tilde M_{j,k}=M_k+\Lambda_j$,
$\tilde h_{j,k}
 = h_k + \Lambda_j^{-1}
   ((\Sigma_j^{(\mathrm{in})})^{-1}m_j^{(\mathrm{in})}
    +\theta_t^{(+)}-a_1^{(+)}m_k-\theta_1^{(+)})$,
and
$\tilde q_{j,k}(t)
 = \tfrac{1}{2}(m_k-\bar d_t)^\top S_k^{-1}(m_k-\bar d_t)
   + \tfrac{1}{2}(m_j^{(\mathrm{in})})^\top
     (\Sigma_j^{(\mathrm{in})})^{-1}m_j^{(\mathrm{in})}$.

\section{Demand-Response Application: Additional Diagnostics}
\label{SI:dr}

\subsection{Experimental parameters}
\label{SI:dr:params}

See Table \ref{table:num-parameters}.

\begin{table}[h!]
\centering
\caption{Numerical parameters for the demand-response simulations.
\label{table:num-parameters}}
\small
\begin{tabular}{ll}
\toprule
Parameter & Value \\
\midrule
Number of particles $B$ & 8{,}000 \\
Euler--Maruyama steps $n_\mathrm{steps}$ & 2{,}500 \\
PWC intervals $M$ & 8 \\
$\beta$-schedule $(\beta_0,\gamma)$ & $(12.0,\;0.65)$ \\
MF tolerance $\epsilon$ & $2\times10^{-4}$ \\
\midrule
Target mixture $(\pi,m^{(\mathrm{tar})},\sigma^{(\mathrm{tar})})$
  & $(0.60,0.40)$; $(0.0,1.5)$; $(0.20,0.30)$ \\
Scenario A $(\sigma^{(\mathrm{in})},m^{(\mathrm{in})})$
  & $(3.0,3.0)$; $(1.0,6.0)$ \\
Scenario B $(\sigma^{(\mathrm{in})},m^{(\mathrm{in})})$
  & $(0.5,0.7)$; $(1.5,5.5)$ \\
\bottomrule
\end{tabular}
\end{table}

Both scenarios use $f_t\doteq  0$.
By Theorem~\ref{thm:linear-guidance}, the MF guidance is therefore
the exact linear interpolant
$\nu_t^{(\MF)} = (1-t)\,\bar m^{(\mathrm{in})} + t\,\bar m^{(\mathrm{tar})}$
in both cases, with no iteration required.
The small deviations $\max|\nu^{(\MF)}-\nu_{\mathrm{lin}}|$
reported in the numerical diagnostics (0.078 for Scenario A,
0.030 for Scenario B; Fig.~\ref{fig:DR_AB_guidance}) are an artefact
of the PWC temporal discretisation: within each of the $M=8$ intervals
the guidance is held constant, so the piecewise-constant approximation
to a linear function is not exactly linear at the midpoints.
These residuals shrink as $M\to\infty$ and are negligible for all
energy comparisons.

\subsection{Scenario A: additional diagnostics}
\label{SI:dr:scenA}

Figures \ref{fig:DR_A_density}--\ref{fig:DR_A_percomp} show density
snapshots, trajectory ensembles, and per-component energy decomposition
for Scenario~A.

\begin{figure}[h!]
  \centering
  \includegraphics[width=\linewidth]{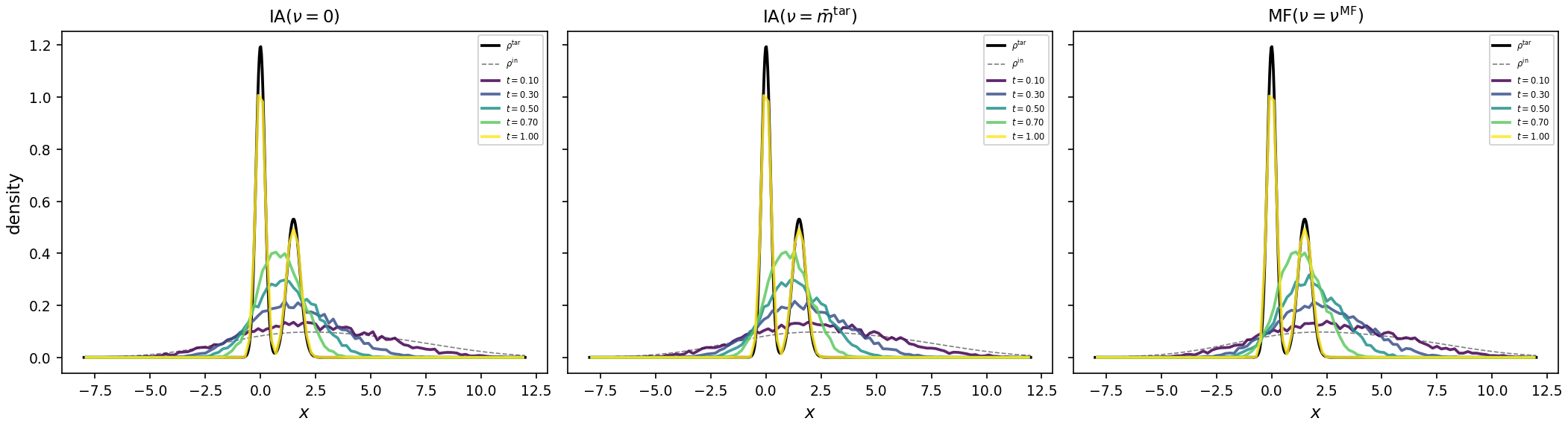}
  \caption{\textbf{Scenario~A}: density snapshots at
    $t\in\{0.10,0.30,0.50,0.70,1.0\}$ for the three methods.
    The initially broad, overlapping distribution contracts and splits
    into the two target modes.
    Black solid: target $\rho^{(\mathrm{tar})}$;
    black dashed: initial $\rho^{(\mathrm{in})}$.}
  \label{fig:DR_A_density}
\end{figure}

\begin{figure}[h!]
  \centering
  \includegraphics[width=\linewidth]{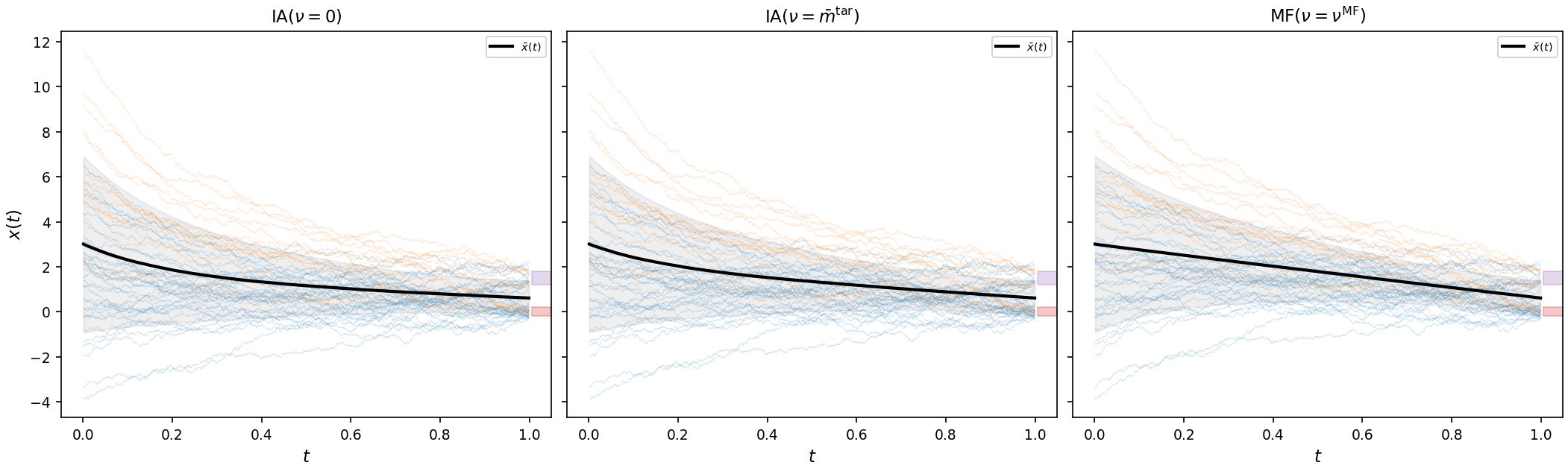}
  \caption{\textbf{Scenario~A}: trajectory ensembles
    (50 sample paths, colored by initial component).
    Blue: occupied; orange: unoccupied.
    Black curve: ensemble mean; gray band: $\pm 1\sigma$.
    Shaded rectangles at $t{=}1$:
    target $\pm\sigma_k^{(\mathrm{tar})}$ bands.}
  \label{fig:DR_A_traj}
\end{figure}

\begin{figure}[h!]
  \centering
  \includegraphics[width=0.8\linewidth]{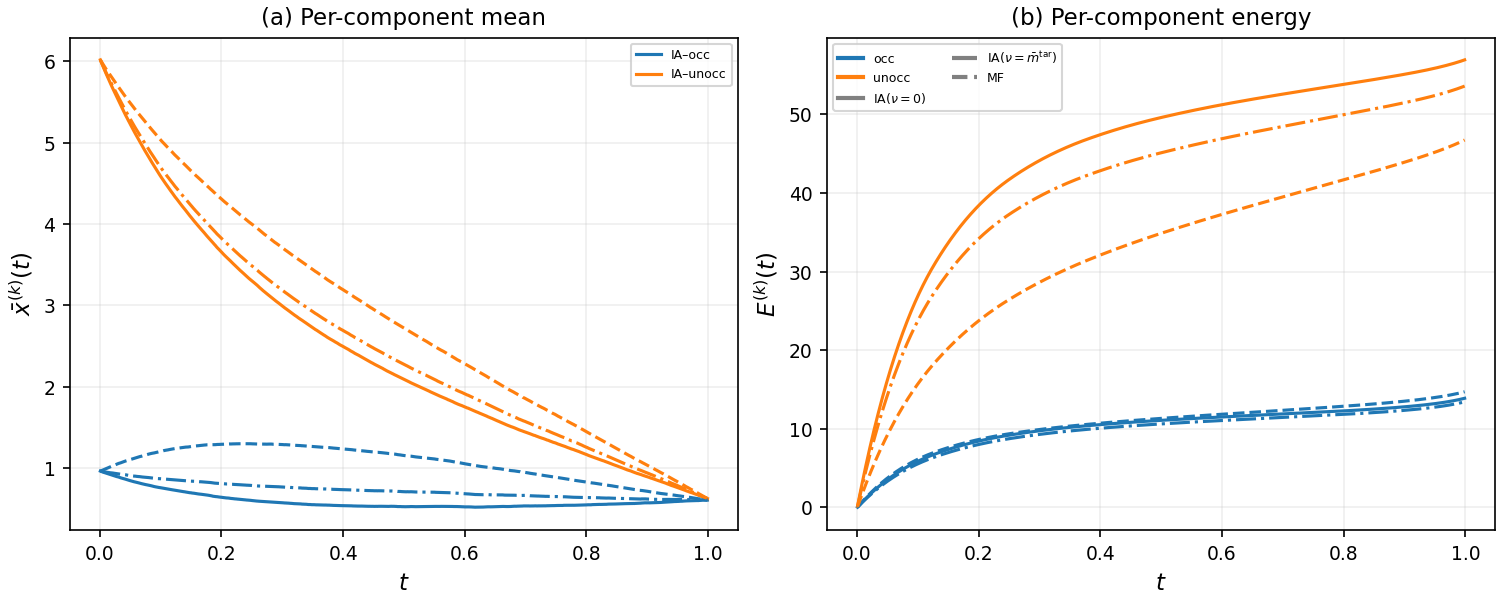}
  \caption{\textbf{Scenario~A}: per-component analysis.
    (a)~Per-component mean trajectories: the unoccupied mode (orange)
    must travel $\sim 4.5$ units while the occupied mode (blue)
    travels $\sim 1.0$.
    (b)~Cumulative per-component energy: the unoccupied mode absorbs
    most of the control cost; MF coordination reduces this asymmetry.}
  \label{fig:DR_A_percomp}
\end{figure}

\subsection{Scenario B: additional diagnostics}
\label{SI:dr:scenB}

Figures \ref{fig:DR_B_sixpanel}--\ref{fig:DR_B_traj} provide the
same diagnostics for Scenario~B.

\begin{figure}[h!]
  \centering
  \includegraphics[width=\linewidth]{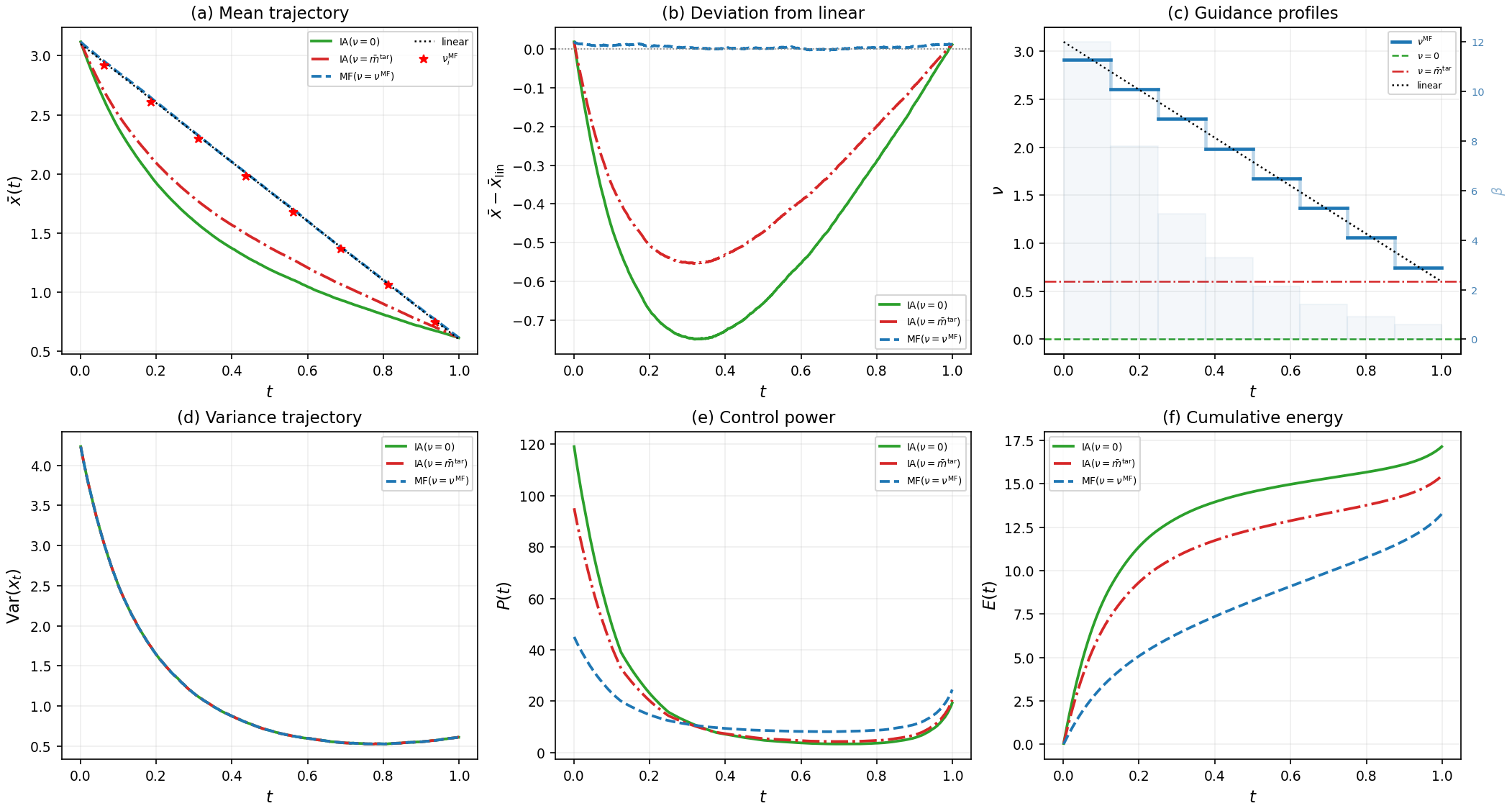}
  \caption{\textbf{Scenario~B} (narrow initial).
    The MF energy saving increases to 22.6\%.}
  \label{fig:DR_B_sixpanel}
\end{figure}

\begin{figure}[h!]
  \centering
  \includegraphics[width=\linewidth]{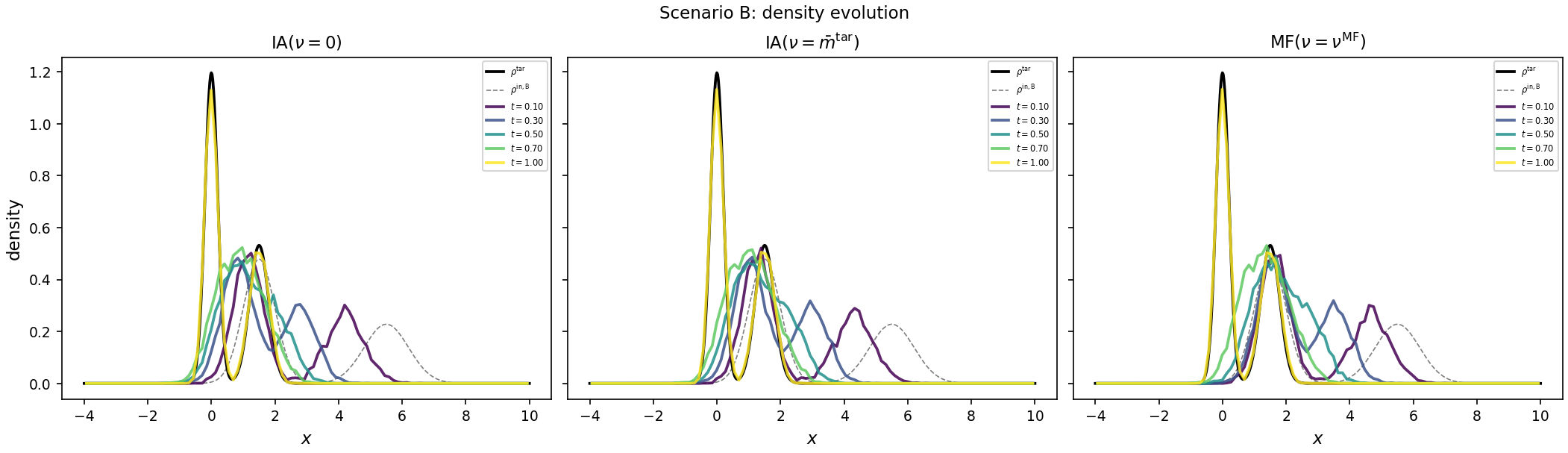}
  \caption{\textbf{Scenario~B}: density snapshots.
    Unlike Scenario~A, the initial law is already bimodal;
    the two peaks translate and sharpen simultaneously.}
  \label{fig:DR_B_density}
\end{figure}

\begin{figure}[h!]
  \centering
  \includegraphics[width=\linewidth]{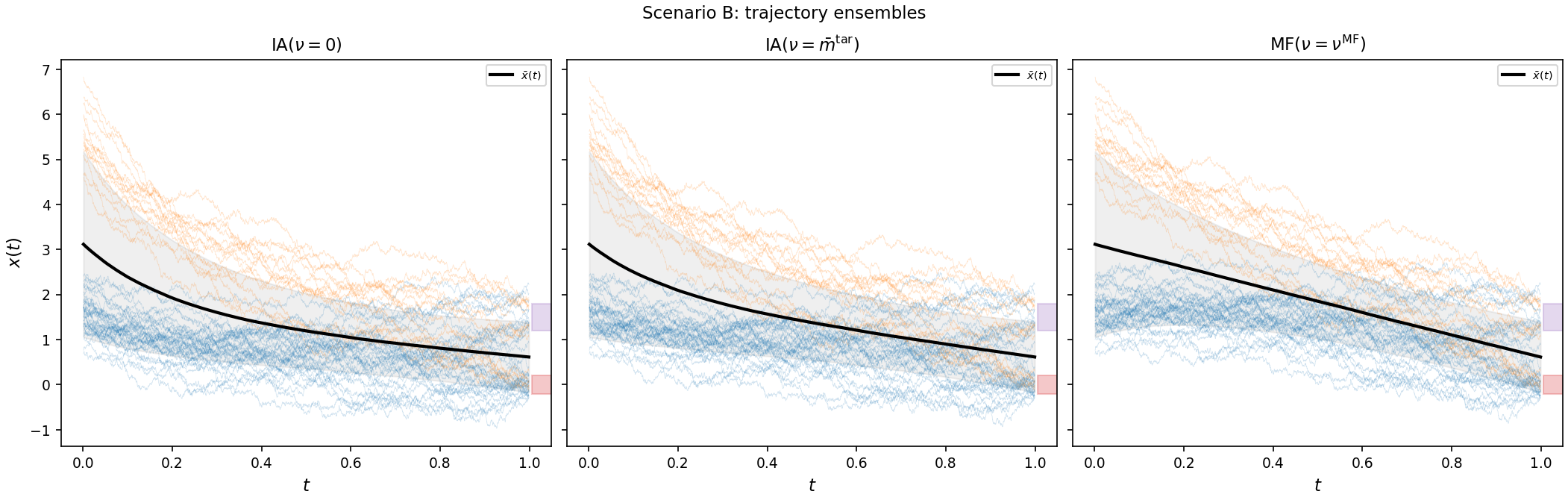}
  \caption{\textbf{Scenario~B}: trajectory ensembles
    (50 sample paths, colored by component).
    The two clusters remain well separated throughout the bridge.}
  \label{fig:DR_B_traj}
\end{figure}

\subsection{Score-field affinity}
\label{SI:dr:affinity}

The LQG benchmark (SI\,\S2) predicts an optimal drift affine in $x$:
$u^*(t,x) = -S(t)\,x - s(t)$.
For Gaussian-mixture targets the drift is no longer exactly affine;
we quantify the departure by fitting
$u^*\approx -S(t)\,x - s(t)$ via weighted least squares at each time
slice and reporting the coefficient of determination $R^2(t)$.

Figure~\ref{fig:DR_Ss_R2} overlays the spring constant $S(t)$,
shift $s(t)$, and affinity $R^2(t)$ for all three methods.

\begin{figure}[h!]
  \centering
  \includegraphics[width=\linewidth]{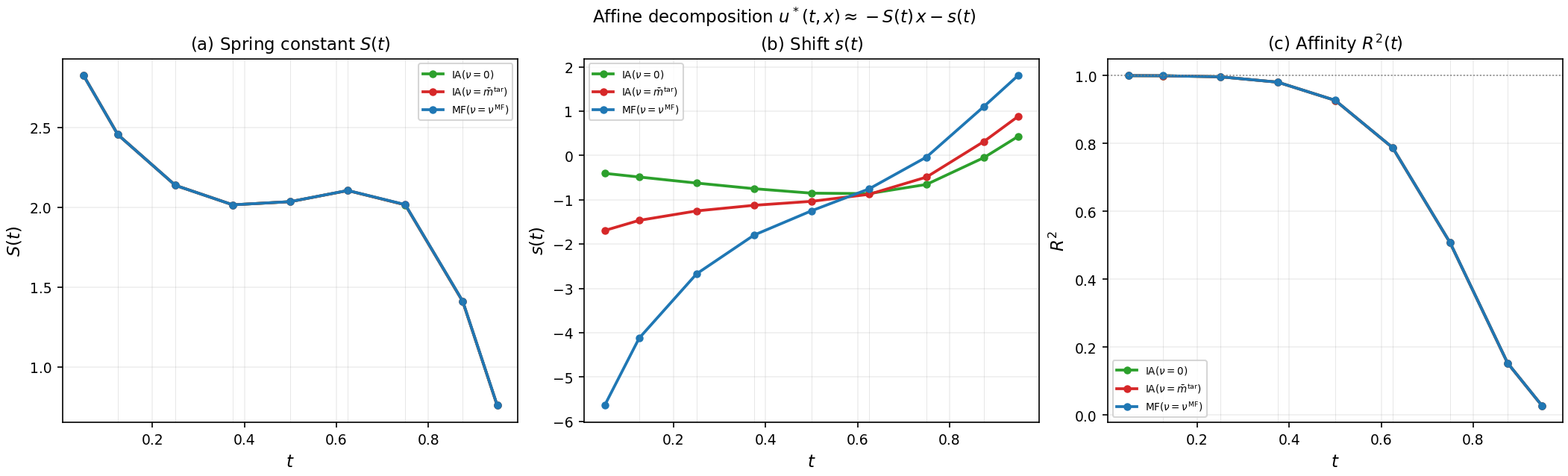}
  \caption{\textbf{Scenario~A}: score-field affine decomposition
    $u^{*}(t,x)\approx -S(t)\,x-s(t)$.
    (a)~Spring constant $S(t)$: identical across methods, set by
    $\beta(t)$ and $\rho^{(\mathrm{tar})}$.
    (b)~Shift $s(t)$: differs between methods;
    $|s^{\mathrm{MF}}|$ is largest at early times.
    (c)~Affinity $R^{2}(t)$: drops from $\sim 1.0$ to $\sim 0.03$,
    marking the onset of genuinely non-affine Gaussian-mixture score
    structure.
    Vertical gray lines: $\beta$-interval boundaries.}
  \label{fig:DR_Ss_R2}
\end{figure}

\subsection{Cross-scenario comparison}
\label{SI:dr:cross}

Figures \ref{fig:DR_AB_energy}--\ref{fig:DR_convergence}
summarise the comparison between scenarios.

\begin{figure}[h!]
  \centering
  \includegraphics[width=0.35\linewidth]{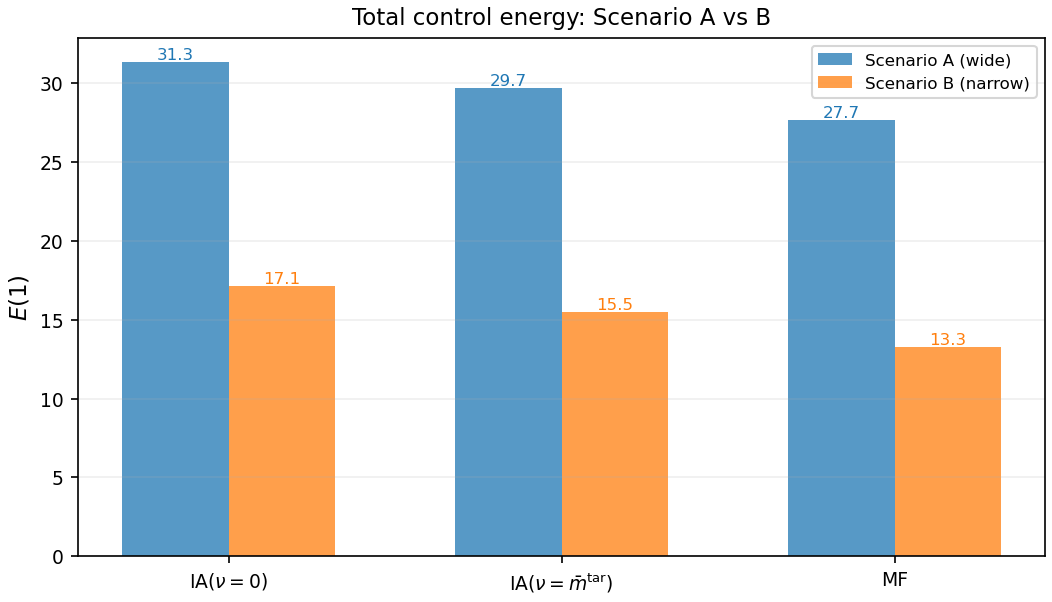}\hfill
  \includegraphics[width=0.6\linewidth]{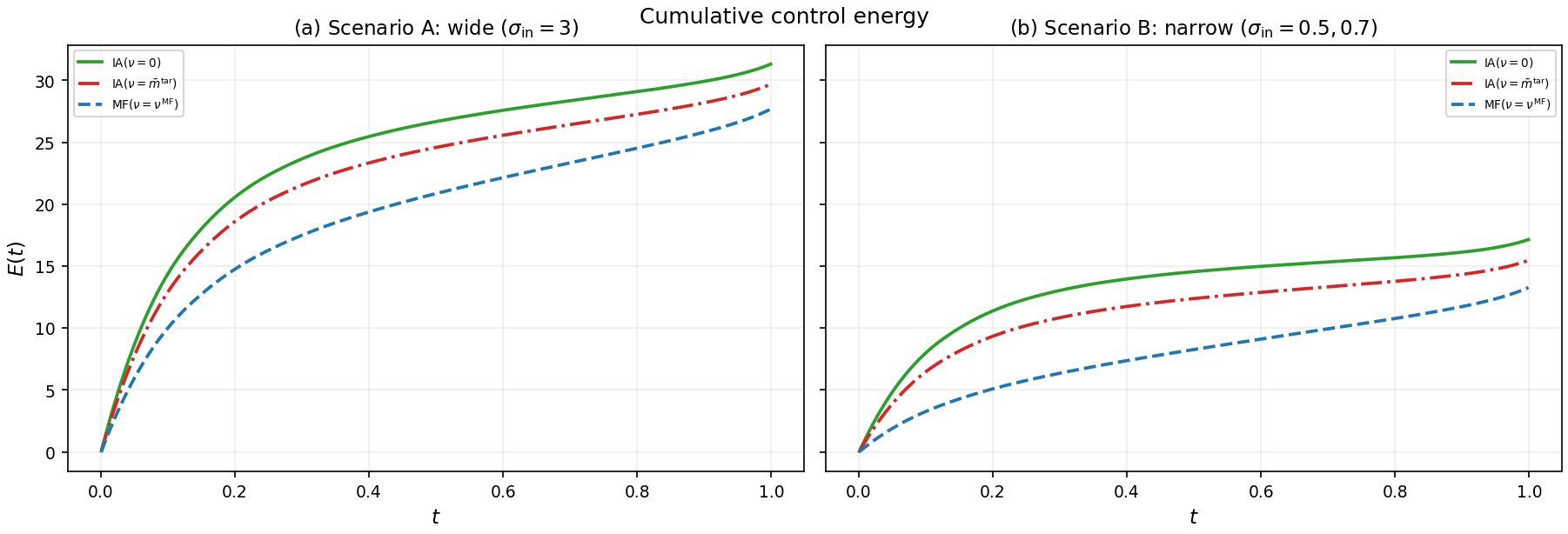}
  \caption{Cross-scenario energy comparison.
    (Left)~Total control energy $\mathcal{E}(1)$ for all three methods
    in Scenarios~A and~B.
    (Right)~Cumulative energy curves.
    The MF saving increases from 11.6\% (A) to 22.6\% (B).}
  \label{fig:DR_AB_energy}
\end{figure}

\paragraph{Energy.}
The absolute energy is roughly halved from Scenario~A to~B
(the narrow initial law starts closer to the target in Wasserstein
distance), but the \emph{relative} MF advantage approximately doubles:
11.6\%~(A) vs.\ 22.6\%~(B).
This amplification arises because, in Scenario~B, the two modes are
quasi-independent: the unoccupied cluster must travel ${\approx}\,4$
units while the occupied cluster moves only ${\approx}\,1.5$, and
the self-consistent guidance of Theorem~\ref{thm:linear-guidance}
adapts to this asymmetry more effectively than any single constant $\nu$.

\begin{figure}[h!]
  \centering
  \includegraphics[width=0.55\linewidth]{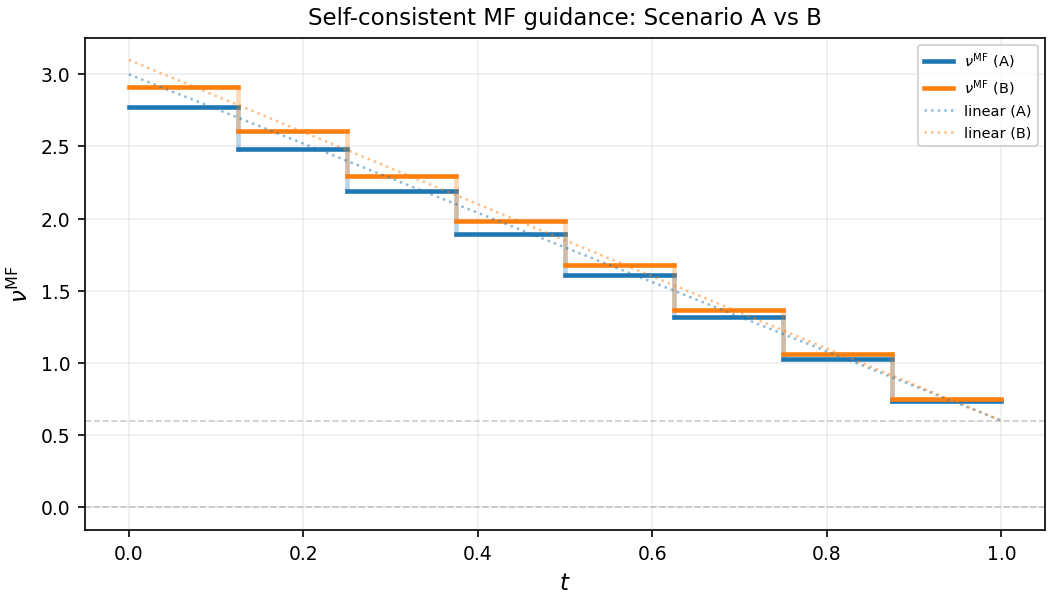}
  \caption{Self-consistent MF guidance $\nu^{(\MF)}(t)$ for
    Scenarios~A (blue) and~B (orange), with their respective linear
    interpolants (dotted).
    By Theorem~\ref{thm:linear-guidance}, the two curves coincide
    exactly in the continuous-time limit; the small residuals
    ($\max|\nu^{(\MF)}-\nu_{\mathrm{lin}}|=0.078$ for~A, $0.030$ for~B)
    are due to the PWC temporal discretisation ($M=8$ intervals).}
  \label{fig:DR_AB_guidance}
\end{figure}

\paragraph{MF guidance.}
Figure~\ref{fig:DR_AB_guidance} confirms Theorem~\ref{thm:linear-guidance}
numerically: the self-consistent guidance $\nu^{(\MF)}(t)$ lies
essentially on the linear interpolant between the initial and target
global means for both scenarios.
The residuals are smaller in Scenario~B, where the modes overlap less
and the PWC approximation to a linear profile is more accurate.

\begin{figure}[h!]
  \centering
  \includegraphics[width=0.48\linewidth]{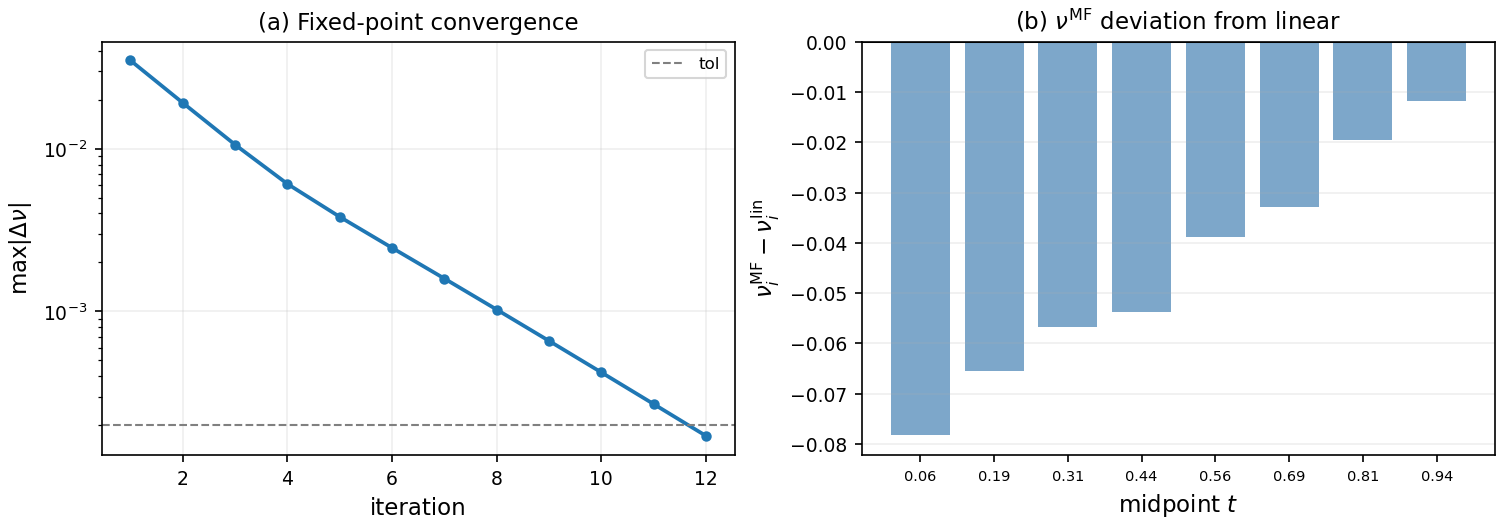}\hfill
  \includegraphics[width=0.48\linewidth]{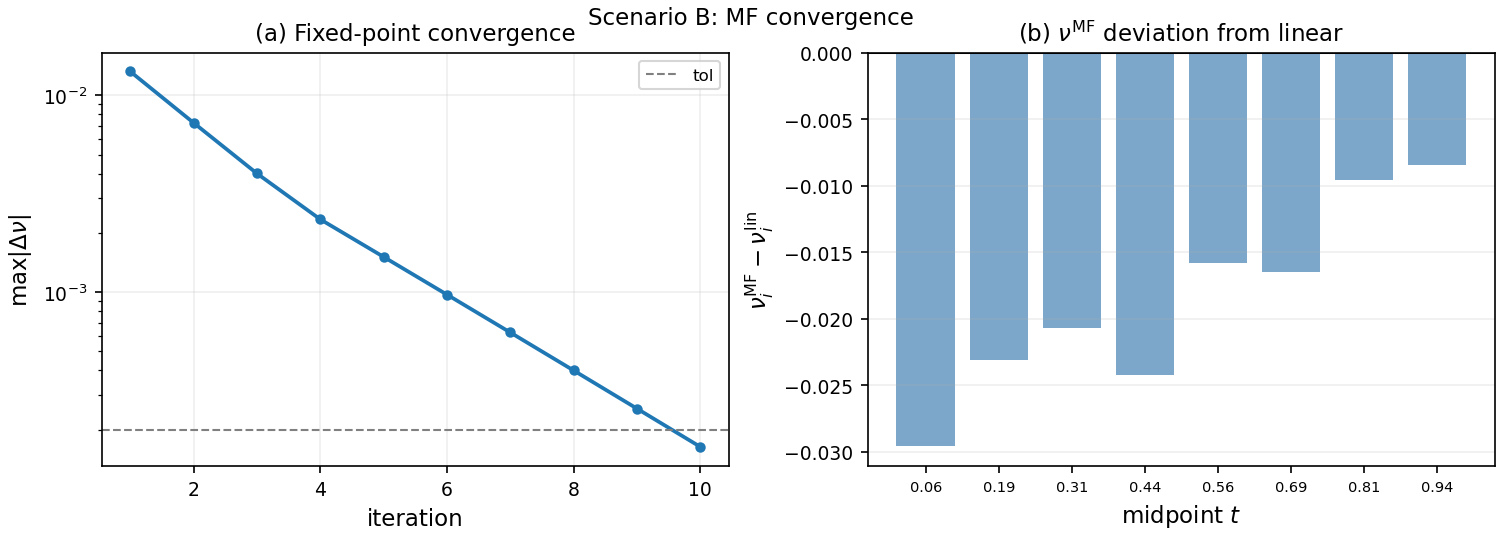}
  \caption{Convergence of the fixed-point iteration to the linear
    interpolant.
    Left: Scenario~A (12 iterations to tolerance $2\times10^{-4}$).
    Right: Scenario~B (10 iterations).
    (a)~$\max|\Delta\nu|$ vs.\ iteration (dashed: tolerance).
    (b)~Residual $\nu_i^{(\MF)}-\nu_i^{\mathrm{lin}}$ at each
    PWC midpoint, confirming Theorem~\ref{thm:linear-guidance}
    in the PWC limit.}
  \label{fig:DR_convergence}
\end{figure}

\section{Independent-Agent PID: Background}
\label{SI:ia}

This section summarises the independent-agent PID construction
\cite{behjoo_harmonic_2025,chertkov_generative_2025} that forms the
foundation for MF-PID.

\subsection{IA PID in the SOT formulation}
\label{SI:ia:sot}

In the absence of inter-agent coupling, the cost-to-go satisfies the
standard HJB equation, and after the Hopf--Cole substitution the
optimal control is
\begin{equation}
  u_t^{(\IA)}(x)
  = \nabla_x\log\int\!\!\ptar(y)\,
      \frac{G_t^{(-)}(x;y)}{G_1^{(+)}(y;0)}\,\dd y,
  \label{SI:eq:u-star-ia}
\end{equation}
where $G_t^{(\pm)}$ satisfy
\eqref{SI:eq:green-minus}--\eqref{SI:eq:green-plus}
with $\Veff_t$ replaced by a prescribed potential $V_t$
(no $p_t$ dependence).
The Green functions can therefore be computed independently of the
population density---this is the key structural simplification that
MF breaks.

\subsection{H-PID with quadratic guidance potential}
\label{SI:ia:hpid}

For $f_t\doteq 0$ and guided quadratic potential
$V_t(x)=\tfrac{\beta_t}{2}\|x-\nu_t\|^2$,
the Green functions take the Gaussian forms
\eqref{SI:eq:G-minus}--\eqref{SI:eq:G-plus} with scalar Riccati
coefficients satisfying \eqref{SI:eq:riccati-abc}.
The score function for a Gaussian-mixture target is given by
Proposition~\ref{prop:gm-score} with a \emph{prescribed} guidance
$\nu_t$.
The PWC closed forms of \S\ref{SI:quad:closedforms} then yield a
fully explicit, training-free generative model.
MF-PID specialises this to the endogenously determined guidance
of Theorem~\ref{thm:linear-guidance}.

\subsection{Cross-validation identity}
\label{SI:ia:cv}

\begin{proposition}[Cross-validation identity]
\label{prop:cv}
The optimal score satisfies
\begin{equation}
  u_t^*(x)
  = \nabla_x\log p_t^*(x) - \nabla_x\log G_t^{(+)}(x;0),
  \label{SI:eq:cross-val}
\end{equation}
relating the optimal control, the optimal marginal density, and the
forward Green function.
\end{proposition}

When $V_t=0$ and $f_t$ is pre-trained to yield $\ptar$ as a stationary
law, $G_1^{(+)}(x;0)=\ptar(x)$, and \eqref{SI:eq:cross-val} implies
$u_t^*\doteq 0$---confirming that no corrective control is needed when
the drift already realises the target.


\end{document}